\documentclass[final, onefignum, onetabnum]{siamart171218}

\usepackage{amsfonts}
\usepackage{multirow}
\usepackage{rotating}
\usepackage{enumerate}
\usepackage{algpseudocode}
\usepackage{tikz}\usetikzlibrary{arrows, decorations.pathreplacing}

\usepackage{advdate}
\AdvanceDate[-113]

\ifpdf
  \DeclareGraphicsExtensions{.eps,.pdf,.png,.jpg}
\else
  \DeclareGraphicsExtensions{.eps}
\fi

\DeclareMathOperator*{\argmin}{arg\, min}
\DeclareMathOperator{\spn}{span}

\newcommand{\mbbN}{\mathbb{N}}
\newcommand{\mbbR}{\mathbb{R}}
\newcommand{\mbbC}{\mathbb{C}}
\newcommand{\mc}[1]{\mathcal{#1}}
\newcommand{\wt}[1]{\widetilde{#1}}
\newcommand{\ol}[1]{\overline{#1}}

\newcommand{\Break}{\State\textbf{break}}
\newcommand{\hypref}[2]{\hyperref[#2]{#1 \ref*{#2}}}
\newcommand{\hyprefp}[2]{(\hyperref[#2]{#1\ref*{#2}})}
\newcommand{\snot}[1]{\mathrm{e}{#1}}

\newsiamremark{remark}{Remark}
\newsiamremark{hypothesis}{Hypothesis}
\crefname{hypothesis}{Hypothesis}{Hypotheses}
\newsiamthm{claim}{Claim}

\headers{Projected Newton method for the Tikhonov-Morozov equations}{Nick Schenkels and Wim Vanroose}

\title{Projected Newton method for a system of Tikhonov-Morozov equations\thanks{Submitted to the editors \today.}}

\author{Nick Schenkels\thanks{Department of Mathematics and Computer Science,
	University of Antwerp (\email{nick.schenkels@uantwerpen.be}).}
\and Wim Vanroose \thanks{Department of Mathematics and Computer Science,
	University of Antwerp (\email{wim.vanroose@uantwerpen.be}).}}

\ifpdf
\hypersetup{
  pdftitle={Projected Newton method for a system of Tikhonov-Morozov equations},
  pdfauthor={Nick Schenkels and Wim Vanroose}
}
\fi

\begin{document}


\maketitle


\begin{abstract}
		In this paper we derive a Newton type method to solve the non-linear system
		formed by combining the Tikhonov normal equations and Morozov's discrepancy
		principle. We prove that by placing a bound on the step size of the Newton
		iterations the method will always converge to the solution. By projecting the
		problem onto a low dimensional Krylov subspace and using the method to solve the projected
		non-linear system we show that we can reduce the computational cost of the method.
\end{abstract}




\section{Introduction}
In this paper we consider linear inverse problems of the form $Ax = b$ with $A\in
\mbbR^{m\times n}$, $x\in\mbbR^n$ and $b\in\mbbR^m$. Here, the right hand side $b$
is the perturbed version of the unknown exact measurements or observations $b_{ex}
= b + e$, with $e\sim\mc{N}(0, \sigma^2I_m)$. It is well known that for ill-posed
problems some form of regularization has to be used in order to deal with the noise
$e$ in the data $b$ and to find a good approximation for the true solution of $Ax = b_{ex}$.
One of the most widely used methods to do so is Tikhonov regularization. In its
standard from, the Tikhonov solution to the inverse problem is given by
\begin{equation}\label{eq:tikhonov}
		x_\alpha = \argmin_{x\in\mbbR^n}\left\|Ax - b\right\|^2 + \alpha\left\|x\right\|^2,
\end{equation}
where $\alpha > 0$ is a regularization parameter and $\left\|\cdot\right\|$ denotes
the standard Euclidean norm. 

The choice of the regularization parameter is very important since its value has
a significant impact on the reconstruction. If, on the one hand, $\alpha$ is chosen
too large, focus lies on minimizing the regularization term $\left\|x\right\|^2$.
The corresponding reconstruction $x_\alpha$ will therefore no longer be a good
solution for the linear system $Ax = b$, will typically have lost many details
and be what is referred to as ``oversmoothed''. If, on the other had, $\alpha$
is chosen too small, focus lies on minimizing the residual $\left\|Ax - b\right\|^2$.
This, however, means that the errors $e$ are not suppressed and that the reconstruction
$x_\alpha$ will be ``overfitted'' to the measurements.

\begin{figure}
		\centering
		\begin{tikzpicture}[line cap = round, line join = round, > = triangle 45, scale = 0.65]
		\begin{scope}[shift = {(-7, 0)}]
				\draw[->] (-1, 0) -- (5, 0);
				\draw[->] (0, -1) -- (0, 5);
				
				 L-curve:
				\draw [shift = {(0, 3)}, thick]  plot[domain = 0:1.07,variable = \t] ({cos(\t r)},{sin(\t r)});
				\draw [shift = {(3, 0)}, thick]  plot[domain = 0.5:1.57,variable = \t] ({cos(\t r)},{sin(\t r)});
				\draw [shift = {(1.5, 1.5)}, thick]  plot[domain = 3.14:4.71,variable = \t] ({cos(\t r)/2},{sin(\t r)/2});
				\draw [thick] (1, 3) -- (1, 1.5);
				\draw [thick] (1.5, 1) -- (3, 1);
				
				\draw [dashed, thick, color = red] (1.75, -0.25) -- (1.75, 3.25);
				\node [anchor = east] at (1.6, 0.35) {\color{red} $\eta\varepsilon$};
				
				\draw [shift = {(0, 3)}] plot[domain = 0.32:1.25, variable = \t] ({1*1.58*cos(\t r)+0*1.58*sin(\t r)},{0*1.58*cos(\t r)+1*1.58*sin(\t r)});
				\draw [shift = {(3, 0)}] plot[domain = 0.32:1.25, variable  =\t] ({1*1.58*cos(\t r)+0*1.58*sin(\t r)},{0*1.58*cos(\t r)+1*1.58*sin(\t r)});
				\draw [->] (0.5, 4.5) -- (0.29, 4.57);
				\draw [->] (4.5, 0.5) -- (4.57, 0.3);
				
				\node [label = below:$\left\|Ax_\alpha - b\right\|$] at (2.5, -0.25) {};
				\node [label = left:\rotatebox{90}{$\left\|x_\alpha\right\|$}] at (-0.25, 2.5) {};
				\node [font = \tiny, label = right:$0\leftarrow\alpha$] at (0.75, 4.5) {};
				\node [font = \tiny, label = right:$\alpha\rightarrow+\infty$] at (4, 1.25) {};
				
				
				\draw[->] (2.75, 2.25) -- (1.8, 1.1);
				\node [align = center] at (4, 2.75) {\footnotesize $\alpha$ for discrepancy\\ \footnotesize principle};
				\draw[->] (-1.5, -0.2) -- (1.1, 1.1);
				\node [align = center] at (-1.5, -1) {\footnotesize $\alpha$ for L-curve\\ \footnotesize method};
		\end{scope}
		
		\begin{scope}[shift = {(1, 0)}]
				\draw[->] (-1, 0) -- (7, 0);
				\draw[->] (0, -1) -- (0, 5);
				
				\draw [thick] (0, 0.25) .. controls (2, 0.25) and (4, 2) .. (6, 4);
				
				\draw [dashed, thick, color = red] (-0.1, 2) -- (7, 2);
				\node [anchor = north] at (7, 1.9) {\color{red} $\eta\varepsilon$};
				
				\node [anchor = north] at (3.5, 0) {$\alpha$};
				\node [anchor = east] at (-0.1, 2.5) {\rotatebox{90}{$\left\|Ax_\alpha - b\right\|$}};
				
				\draw[->] (3.3, 3.1) -- (3.8, 2.1);
				\node [align = center] at (2.75, 3.75) {\footnotesize $\alpha$ for discrepancy\\  \footnotesize principle};
				
				\draw [thick, decorate, decoration={brace, mirror}] (0, -0.6) -- (3.7, -0.6)
						node [midway, yshift = -10] {\footnotesize Overfitting};
				\draw [thick, decorate, decoration={brace, mirror}] (3.8, -0.6) -- (7, -0.6)
						node [midway, yshift = -10] {\footnotesize Oversmoothing};
		\end{scope}
\end{tikzpicture}
		\caption{Sketch of the L-curve (left) and the D-curve (right). The value
		for $\alpha$ proposed by the L-curve method is typically slightly larger
		than the one proposed by the discrepancy principle \cite{hansen1992}.}
		\label{fig:curves}
\end{figure}
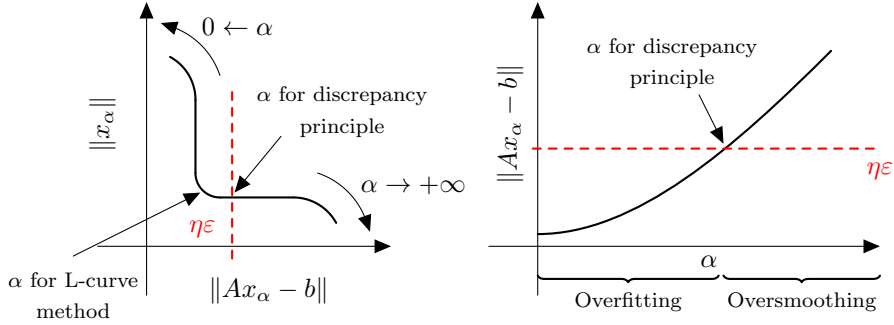

One way of choosing the regularization parameter is the L-curve method. If
$x_\alpha$ is the solution of the Tikhonov problem \eqref{eq:tikhonov}, then
the curve $(\left\|Ax_\alpha - b\right\|, \left\|x_\alpha\right\|)$ typically
has a rough ``L'' shape, see \hypref{figure}{fig:curves}. Heuristically, the
value for the regularization parameter corresponding to the corner of this ``L''
has been proposed as a good regularization parameter because is balances model
fidelity (minimizing the residual) and regularizing the solution (minimizing the
regularization term) \cite{calvetti1999, hansen1992, hansen1993, hansen2010}.
The problem with this method is that in order to find this value, the Tikhonov
problem has to be solved for many different values of $\alpha$, which can be
computationally expensive and inefficient for large scale problems.

Another way of choosing the regularization parameter is Morozov's discrepancy
principle \cite{morozov1984}. Here, the regularization parameter is chosen such that
\begin{equation}\label{eq:morozov}
		\left\|Ax_\alpha - b\right\| = \eta\varepsilon
\end{equation}
with $\varepsilon = \left\|e\right\|$ the size of the error and $1\leq\eta$
a tolerance value. The idea behind this choice is that finding a solution $x_\alpha$
with a lower residual can only lead to overfitting. Similarly to the L-curve,
we can look at the curve $(\alpha, \left\|Ax_\alpha - b\right\|)$, which we'll
refer to as the discrepancy curve or D-curve, see \hypref{figure}{fig:curves}.
If $e\sim\mc{N}(0, \sigma^2I_m)$, then it is an easy verification to see that
$\varepsilon\approx\sigma\sqrt{m}$, but in general the size of the error may
be unknown.

In this paper we describe a Newton type method that simultaneously updates the
solution $x$ and the regularization parameter $\alpha$ such that the Tikhonov
problem \eqref{eq:tikhonov} and Morozov's discrepancy principle \eqref{eq:morozov}
are both satisfied. This is done by combining both equations into one big non-linear
system in $x$ and $\alpha$ and solving it using Newton's method. However, starting
from an arbitrary initial estimate, convergence of the classical Newton's method
cannot be guaranteed. In \hypref{section}{sec:ntm} we prove that by starting from
a specific initial estimate and placing a bound on the step size of the Newton
updates the method will always converge. We also derive an estimate for this step
size. For large scale problems computing the Newton search directions and this
step size can, however, be computationally expensive. In \hypref{section}{sec:pntm}
we therefore combine our method with a projection onto a low dimensional Krylov
subspace. In \hypref{sections}{sec:numexp1} and \ref{sec:numexp2} we perform
extensive numerical experiments in order to illustrate the workings of these
methods and compare them with other regularization methods found in the literature,
see \hypref{section}{sec:refmethods}. Finally, in \hypref{section}{sec:concl},
we end the paper with a short discussion on some open questions that remain.


\section{Tikhonov-Morozov system}\label{sec:ntm}
In order to find $(x, \alpha)\in\mbbR^n\times\mbbR_0^+$ that solves the Tikhonov
problem and satisfies the discrepancy principle, we consider the non-linear system
\begin{equation}\label{eq:tikmor}
		\left\{\begin{aligned}
				F_1(x, \alpha) &= (A^T A + \alpha I)x - A^Tb\\
				F_2(x, \alpha) &= \frac{1}{2}(Ax - b)^T(Ax - b) - \frac{1}{2}\varepsilon^2
		\end{aligned}\right.
\end{equation}
for $F:\mbbR^n\times\mbbR_0^+\longmapsto\mbbR^{n}\times\mbbR_0^+$. Here, $F_1(x, \alpha)
= 0$ are the normal equations corresponding to the Tikhonov problem \eqref{eq:tikhonov}
with regularization parameter $\alpha$ and $F_2(x, \alpha) = 0$ is equivalent to Morozov's
discrepancy principle \eqref{eq:morozov} (for simplicity we assume that $\eta = 1)$.

If we apply Newton's method to solve this non-linear system of equations, convergence
of the method starting from an arbitrary initial estimate cannot be guaranteed.
We will prove that by starting from a point $(x_0, \alpha_0)$ satisfying the Tikhonov
normal equations $F_1$, we can guarantee convergence of Newton's method by limiting
the step size. The idea behind this approach is the observation that for points
which ``almost'' satisfy these equations, the Jacobian will be invertible. By
placing a bound on the Newton step size, we can force the iterations to remain
within this region of interest and prove convergence.


\subsection{Newton iterations}
If the current Newton iteration for the solution of \eqref{eq:tikmor} is given
by $\left(x_{k - 1}, \alpha_{k - 1}\right)$, then we write the next iteration as
\[
		x_k = x_{k - 1} + \Delta x_k\qquad\text{and}\qquad \alpha_k = \alpha_{k - 1}
				+ \Delta \alpha_k.
\]
The Jacobian system for the Newton search directions is now given by
\[
		\begin{pmatrix}A^TA+ \alpha_{k - 1} I  & x_{k - 1}\\ (Ax_{k - 1} - b)^T A
				& 0\end{pmatrix}\begin{pmatrix}\Delta x_k\\	\Delta\alpha_k\end{pmatrix}
				= -\begin{pmatrix}(A^TA + \alpha_{k - 1} I)x_{k - 1} - A^T b\\ \frac{1}
				{2}(Ax_{k - 1} - b)^T(Ax_{k - 1} - b)-\frac{1}{2}\epsilon^2
		\end{pmatrix},
\]
or in short
\begin{equation}\label{eq:newtoneq}
  J(x_{k - 1}, \alpha_{k - 1})\begin{pmatrix}\Delta x_k\\ \Delta\alpha_k\end{pmatrix}
			= -F(x_{k - 1}, \alpha_{k - 1}).
\end{equation}
\begin{lemma}
		For all Newton iterations with $k\in\mbbN_0$, the following relationship holds:
		\[
				F(x_k, \alpha_k) = \begin{pmatrix}\Delta\alpha_k\Delta x_k\\
						\frac{1}{2}\Delta x_k^TA^TA\Delta x_k\end{pmatrix}.
		\]
\end{lemma}
\begin{proof}
		Using the definition of $F$, it is a straightforward calculation to find that
		\begin{align*}
				F(x_k, \alpha_k) =& F(x_{k - 1} + \Delta x_k, \alpha_{k - 1} + \Delta\alpha_k)\\
				=& J(x_{k - 1}, \alpha_{k - 1})\begin{pmatrix}\Delta x_k\\ \Delta\alpha_k
						\end{pmatrix}	+ F(x_{k - 1}, \alpha_{k - 1}) + \begin{pmatrix}\Delta
						\alpha_k\Delta x_k\\ \frac{1}{2}\Delta x_k^TA^TA_k\Delta x\end{pmatrix}.
		\end{align*}
		Because the search directions $\Delta x_k$ and $\Delta\alpha_k$ are
		found by solving \eqref{eq:newtoneq}, the sum of first two terms equals
		zero, proving the lemma.
\end{proof}

This lemma implies that
\begin{equation}\label{eq:temp}
		J(x_k, \alpha_k)\begin{pmatrix}\Delta x_{k + 1}\\ \Delta\alpha_{k + 1}
				\end{pmatrix}	= -\begin{pmatrix}\Delta\alpha_k & 0\\ \frac{1}{2}
				\Delta x_k^TA^TA & 0\end{pmatrix}\begin{pmatrix}\Delta x_k\\
				\Delta\alpha_k\end{pmatrix},
\end{equation}
resulting in a recurrence relation between two sequential Newton search directions.
Another consequence of the lemma is that
\begin{equation}\label{eq:temp2}
		\begin{aligned}
				&&\left(A^TA + \alpha_kI\right)x_k - A^Tb &= \Delta\alpha_k\Delta x_k\\
				\Leftrightarrow&& A^T\left(Ax_k - b\right) &= -\alpha_kx_k + \Delta\alpha_k
						\Delta x_k.
		\end{aligned}
\end{equation}
This means that if we rescale the last row of \eqref{eq:newtoneq} with $\alpha_{k - 1} > 0$
and instead solve
\[
		\begin{aligned}
				&\begin{pmatrix} A^TA + \alpha_{k - 1} I  & x_{k - 1}\\ \frac{1}{\alpha_{k - 1}}
						(Ax_{k - 1} - b)^TA  & 0 \end{pmatrix}\begin{pmatrix}	\Delta x_k\\
						\Delta\alpha_k \end{pmatrix}\\
				&\hspace{7em}= -\begin{pmatrix} (A^TA + \alpha_{k - 1} I)x_{k - 1} - A^Tb\\ \frac{1}
						{2\alpha_{k - 1}}(Ax_{k - 1} - b)^T(Ax_{k - 1} - b) - \frac{1}{2
						\alpha_{k - 1}}\epsilon^2
				\end{pmatrix},
		\end{aligned}
\]
then the same search directions are found and \eqref{eq:temp} and \eqref{eq:temp2}
remain valid.


\subsection{At the discrepancy curve}
Assume we have $\alpha > 0$ and $x$ such that $F_1(x, \alpha) = 0$. This means
that $x$ is the solution of the Tikhonov normal equations
\[
		(A^TA + \alpha I)x = A^Tb\ \Leftrightarrow\ \frac{1}{\alpha}(Ax - b)^TA = -x^T
\]
and $\left(\alpha, \left\|Ax - b\right\|\right)$ is a point on the discrepancy
curve, but not necessarily corresponding to the optimal value of the regularization
parameter. In this case, the rescaled Jacobian matrix for the Newton system has the
following simplified form:
\[
		D(x, \alpha) := \begin{pmatrix}
      A^T A + \alpha I   &x \\
      -x^T  & 0 
    \end{pmatrix}.
\]
We now look at the numerical range \cite{givens1952}, which for a matrix $A\in
\mathbb{C}^{n\times n}$ is defined as
\[
		W(A) = \left\{\left.\frac{x^*Ax}{x^*x}\ \right|\ x\in\mbbC^n, x\neq 0 \right\},
\]
where $x^*$ denotes the complex conjugate of $x$. This is a useful tool since it
contains the spectrum of the matrix and for $D(x, \alpha)\in\mbbR^{(n + 1)\times (n + 1)}$
we find that
\[
		\begin{pmatrix} u^* & v^* \end{pmatrix}\begin{pmatrix} A^TA + \alpha I & x\\
				-x^T & 0 \end{pmatrix}\begin{pmatrix} u\\ v \end{pmatrix} = u^*(A^TA + \alpha I)u
				+ u^*x v - v^* x^T u
\]
with $u\in\mbbC^n$, $v\in\mbbC$ and $\left(u^T, v^T\right)^T \neq 0$.
Since $\alpha > 0$ and $x\in\mbbR^n$, the first term is strictly positive and real
and the last two terms add up to a pure imaginary number. This means that $\forall z
\in W(D): real(z) > 0$, implying that $0$ is not an eigenvalue and hence that $D$ is
invertible.
\begin{lemma}\label{lem:normDinv}
		For any matrix $A\in\mbbR^{m\times n}$, vector $x\in\mbbR^n$ and $\alpha > 0$
		the Schur complement of $D(x, \alpha)$ exists and is given by $s = x^T(A^TA +
		\alpha I)^{-1}x\in\mbbR$. If we set $t := (A^TA + \alpha I)^{-1}x\in\mbbR^n$,
		then it follows that the inverse of $D$ is given by
		\[
				D^{-1}(x, \alpha) = \begin{pmatrix} (A^TA + \alpha I)^{-1} - \frac{t^Tt}{s} &
						-\frac{t}{s}\\ \frac{t^T}{s} & \frac{1}{s}\end{pmatrix}
		\]
		and that the norm of this matrix is bounded:
		\begin{equation}\label{eq:normDinv}
				\left\|D^{-1}\right\|\leq\left(1 + \frac{\left\|x\right\|}{\alpha}\right)^2
						\max\left\{\frac{1}{\alpha}, \frac{\alpha + \lambda_1}{\left\|x\right\|}
						\right\}.
		\end{equation}
		Here, $\lambda_1$ is the largest eigenvalue of $A^TA$.
\end{lemma}
\begin{proof}
		First note that since $A^TA$ is positive semi-definite, the eigenvalues are given by
		$\lambda_1\geq\lambda_2\geq\ldots\geq\lambda_n\geq 0$. This means that $\left(A^TA + \alpha I\right)$
		is invertible because it has eigenvalues $\lambda_1 + \alpha\geq\lambda_2 + \alpha\geq\ldots\geq\lambda_n
		+ \alpha > 0$. As a result, the Schur complement of $D$ exists and the formula
		for $D^{-1}$ can easily be verified, see for example \cite{zhang2006}. It now also
		follows that the eigenvalues of $\left(A^TA + \alpha I\right)^{-1}$ are given by
		\[
				\frac{1}{\alpha + \lambda_n}\geq\frac{1}{\alpha + \lambda_{n - 1}}\geq\ldots
						\geq\frac{1}{\alpha + \lambda_1} > 0
		\]
		and thus that
		\[
				\left\|\left(A^TA + \alpha I\right)^{-1}\right\|\leq\frac{1}{\alpha +
						\lambda_n}\leq\frac{1}{\alpha}
		\]
		and 
		\[
				\left\|t\right\|\leq\frac{\left\|x\right\|}{\alpha}\quad\text{and}\quad
						\left\|s\right\|\leq\frac{\left\|x\right\|^2}{\alpha}.
		\]
		We now write
		\[
				D^{-1} = \begin{pmatrix}I & -t\\0 & I\end{pmatrix}\begin{pmatrix}\left(
						A^TA + \alpha I\right)^{-1} & 0\\0 & \frac{1}{s}\end{pmatrix}
						\begin{pmatrix}I & 0\\t^T & I\end{pmatrix}
		\]
		and will estimate a bound on the norm of all three matrices. For the first
		matrix we find that for any unit vector $\left(u^T, v^T\right)^T\in\mbbR^n\times\mbbR$:
		\begin{align*}
				\left\|\begin{pmatrix}I & -t\\0 & I\end{pmatrix}\begin{pmatrix}u\\v\end{pmatrix}
						\right\| &= \left\|\begin{pmatrix}u - tv\\v\end{pmatrix}\right\|\\
				&\leq\sqrt{u^Tu - 2u^Ttv + v^2t^Tt + v^2}\\
				&\leq\sqrt{1 + 2\left\|t\right\| + \left\|t\right\|^2}\\
				&\leq\sqrt{\left(1 + \left\|t\right\|\right)^2}\\
				&\leq 1 + \frac{\left\|x\right\|}{\alpha}
		\end{align*}
		Analogously, the same bound can be found for the third matrix. For the
		second matrix we have that
		\[
				\left\|\begin{pmatrix}\left(A^TA + \alpha I\right)^{-1} & 0\\ 0 & \frac{1}{s}
						\end{pmatrix}\right\| = \max\left\{\left\|\left(A^TA + \alpha I\right)^{-1}
						\right\|, \frac{1}{s}\right\}.
		\]
		It now follows from the min-max theorem \cite{wilkinson1965} that
		\[
				s = x^T\left(A^TA + \alpha I\right)^{-1}x\geq\frac{\left\|x\right\|}
						{\alpha + \lambda_1}.
		\]
		Combining all these results proves the lemma.
\end{proof}


\subsection{Step size}
We already showed that for points on the discrepancy curve, the inverse Jacobian
exists and has a bounded norm. However, even when we start from a point on the
discrepancy curve, there is no guarantee that the Newton iterations will remain
on this curve. Hence, we are not certain that the linear systems for the Newton update
will not become singular. In order to avoid this, we will consider two conditions which are
sufficient for the Newton iterations to converge:
\begin{enumerate}[\indent(C1)]
		\item The inverse Jacobian exists in the next iteration $(x_k, \alpha_k)$.\label{c1}
		\item The size of the Newton search direction $\left\|\left(\Delta x_k^T, \Delta\alpha_k
				\right)^T\right\|$ decreases.\label{c2}
\end{enumerate}
We now show that by placing a bound on the step size of the Newton iterations
both conditions can be fulfilled.

In order to derive this bound, we write the Jacobian in any point as a perturbed
version of the matrix $D$ using \eqref{eq:temp2}:
\begin{equation}\label{eq:JDplusE}
		\begin{aligned}
				J(x_k, \alpha_k) =& \begin{pmatrix} A^TA + \alpha_kI & x_k \\ -x_k +
						\frac{\Delta\alpha_k}{\alpha_k}\Delta x_k & 0 \end{pmatrix}\\
				=& \begin{pmatrix} A^T A + \alpha_{k - 1}I & x_{k - 1} \\ -x_{k - 1}^T
						& 0 \end{pmatrix} + \begin{pmatrix} \Delta\alpha_k I & \Delta x_k \\
						-\frac{\alpha_{k - 1}}{\alpha_{k - 1} + \Delta\alpha_k}\Delta x_k^T & 0\end{pmatrix}.
		\end{aligned}
\end{equation}
We also replace the Newton updates with a scaled version
\[
		x_k = x_{k - 1} + \gamma_k\Delta x_k\qquad\text{and}\qquad
				\alpha_k = \alpha_{k - 1} + \gamma_k\Delta\alpha_k
\]
with
\[
		\gamma_k\in I_k := \left\{\begin{aligned}
				&\left]0, 1\right] && \text{if } \Delta\alpha_k > 0\\
				&\left]0, 1\right] && \text{if } \Delta\alpha_k < 0\text{ and }\alpha_{k - 1}
						+ \Delta\alpha_k > 0\\
				&\left]0, -\omega\alpha_{k - 1}/\Delta\alpha_k\right] &&
						\text{if } \Delta\alpha_k < 0\text{ and }\alpha_{k - 1} + \Delta\alpha_k < 0
		\end{aligned}\right.
\]
and a tolerance value $\omega\in]0, 1[$. This is to ensure that the iterates for
$\alpha_k$ remain positive and the reason why we consider three different cases will
become clear in \hypref{lemma}{lem:theta}. This means that \eqref{eq:JDplusE} becomes
\[
		J(x_k, \alpha_k) = \underbrace{\begin{pmatrix} A^T A + \alpha_{k - 1}I & x_{k - 1} \\
				-x_{k - 1}^T & 0 \end{pmatrix}}_{D_{k - 1}^{-1}:=} + \underbrace{\gamma_k\begin{pmatrix}
				\Delta\alpha_k I & \Delta x_k \\ -\zeta_k\Delta x_k^T & 0\end{pmatrix}}_{E_k:=}.
\]
with $\zeta_k = \alpha_{k - 1}/(\alpha_{k - 1} + \gamma_k\Delta\alpha_k)$. We
also define the matrix
\[
		M_k := \gamma_k\begin{pmatrix} \Delta\alpha_kI & 0 \\ \frac{1}{2}\Delta
				x_k^TA^T A & 0 \end{pmatrix}.
\]
Note that we have already shown that $D_{k - 1} = D(x_{k - 1}, \alpha_{k - 1})$
has a bounded inverse, so we can use the following theorem:
\begin{theorem}[Trefethen and Embree]\label{thm:tref}
		Suppose D has a bounded inverse $D^{-1}$, then for any $E$ with $\left\|E\right\| < 
		1/\left\|D^{-1}\right\|$, $D + E$ has a bounded inverse $(D + E)^{-1}$ satisfying
		\[
				\left\|(D + E)^{-1}\right\|\leq\frac{\left\|D^{-1}\right\|}{1 - \left\|E\right\|
						\left\|D^{-1}\right\|}
		\]
		Conversely, for any $\mu > 1/\left\|D^{-1}\right\|$, there exists an $E$ with
		$\left\|E\right\| < \mu$ such that $(D + E)u = 0$ for some non zero $u$.
\end{theorem}
\begin{proof}
		For a proof of this theorem we refer to \cite[p. 28]{trefethen2005}.
\end{proof}

\begin{lemma}\label{lem:normEandM}
		For the matrices $E_k, M_k\in\mbbR^{(n + 1)\times(n + 1)}$ defined above,
		the following holds:
		\begin{align*}
				\left\|E_k\right\| &\leq \gamma_k\left(\left|\Delta\alpha_k\right| +
						\sqrt{1 + \zeta_k^2}\left\|\Delta x_k\right\|\right)\\
				\left\|M_k\right\| &= \gamma_k\sqrt{\Delta\alpha_k^2 + \frac{1}{4}\left\|
						A^TA\Delta x_k\right\|^2}.
		\end{align*}
		As a consequence we have that
		\[
				\lim_{\gamma_k\rightarrow 0}\left\|E_k\right\| = 0\qquad\text{and}
						\qquad\lim_{\gamma_k\rightarrow 0}\left\|M_k\right\| = 0.
		\]
\end{lemma}
\begin{proof}
		Using the triangle inequality we find that
		\[
				\left\|E_k\right\|\leq\gamma_k\left(\left\|\begin{pmatrix} \Delta\alpha_kI &
						0 \\ 0 & 0 \end{pmatrix}\right\| + \left\|\begin{pmatrix} 0 & \Delta x_k \\
						-\zeta_k\Delta x_k^T & 0 \end{pmatrix}\right\|\right).
		\]
		The first matrix is a diagonal matrix with entries $\Delta\alpha_k$ and $0$,
		hence its norm is equal to $\left|\Delta\alpha_k\right|$. For the second matrix
		we take $u\in\mbbR^n$ and $v\in\mbbR$ and find that
		\begin{align*}
				\left\|\begin{pmatrix} 0 & \Delta x_k \\ -\zeta_k\Delta x_k^T & 0 \end{pmatrix}
						\begin{pmatrix} u \\ v \end{pmatrix}\right\| &= \left\|\begin{pmatrix}
						\Delta x_kv \\ -\zeta_k\Delta x_k^Tu\end{pmatrix}\right\| =
						\sqrt{\left\|\Delta x_kv\right\|^2 + \left\|\zeta_k\Delta x_k^Tu\right\|^2}\\
				&\leq \sqrt{v^2 + \zeta^2\left\|u\right\|^2}\left\|\Delta x_k\right\|\\
				\Rightarrow\left\|\begin{pmatrix} 0 & \Delta x_k \\ -\zeta_k\Delta x_k^T & 0
						\end{pmatrix}\right\| &\leq \sqrt{1 + \zeta^2}\left\|\Delta x_k\right\|.
		\end{align*}
		The statement about $\left\|E_k\right\|$ now follows. Similarly, we find for
		$M_k$ that
		\begin{align*}
				\left\|M_k\begin{pmatrix}u \\ v\end{pmatrix}\right\| &= \left\|
						\gamma_k\begin{pmatrix}\Delta\alpha_ku \\ \frac{1}{2}\Delta
						x_k^TA^TAu\end{pmatrix}\right\| = \gamma\sqrt{\Delta\alpha_k^2\left\|
						u\right\|^2 + \frac{1}{4}\left\|\Delta x_k^TA^TAu\right\|^2}\\
				&\leq \gamma\sqrt{\Delta\alpha_k^2 + \frac{1}{4}\left\|A^TA\Delta x_k
						\right\|^2}\left\|u\right\|
		\end{align*}
		By taking $(u, v) = \frac{1}{\left\|A^TA\Delta x_k\right\|}(A^TA\Delta x_k, 0)$ this is an
		equality, proving the statement about $\left\|M_k\right\|$. Finally, it should
		be noted that $\lim_{\gamma_k\rightarrow 0}\zeta_k = 1$, so for $\gamma_k\rightarrow 0$,
		the norms of both matrices also go to $0$.
\end{proof}

\begin{theorem}\label{thm:gamma}
		Starting from an initial point $(x_0, \alpha_0)$ satisfying the Tikhonov normal
		equations $F_1(x_0, \alpha_0) = 0$, there exist $\gamma_k\in I_k$ such that
		\begin{align}
				&\left\|E_k\right\|\left\|D_{k - 1}^{-1}\right\| < 1\label{eq:constr1}\\
				&\frac{\left\|D_{k - 1}^{-1}\right\|}{1 - \left\|E_k\right\|\left\|D_{k - 1}^{-1}\right\|}
						\left\|M_k\right\| < 1\label{eq:constr2}.
		\end{align}
		Scaling the Newton search direction with such a step size $\gamma_k$ is sufficient
		for the Newton iterations to converge.
\end{theorem}
\begin{proof}
		If \eqref{eq:constr1} holds, then it follows from \hypref{theorem}{thm:tref}
		that the inverse Jacobian $J^{-1}(x_k, \alpha_k) = \left(D_{k - 1} + E_k\right)^{-1}$ exists,
		fulfilling condition \hyprefp{C}{c1}. Furthermore, from the recursion between the
		Newton updates \eqref{eq:temp} it also follows that
		\[
				\left\|J^{-1}(x_k, \alpha_k)\begin{pmatrix}\Delta\alpha_kI & 0\\
						\frac{1}{2}\Delta x_k^TA^T A & 0\end{pmatrix}\right\| < 1.
		\]
		is a sufficient condition for \hyprefp{C}{c2} to hold. \eqref{eq:constr2} is simply a
		stronger version of this condition using the bound on $\left\|J^{-1}
		(x_k, \alpha_k)\right\|$ given by \hypref{theorem}{thm:tref}.
		
		It now remains to be shown that such a $\gamma_k$ always exists. Since
		\eqref{eq:constr2} is equivalent to
		\[
				\left\|M_k\right\|\left\|D_{k - 1}^{-1}\right\| < 1 - \left\|E_k\right\|
						\left\|D_{k - 1}^{-1}\right\|
		\]
		and the left hand side is positive, \eqref{eq:constr1} is implied by \eqref{eq:constr2}.
		Also, since the left hand side goes to $0$ when
		$\gamma_k\rightarrow 0$ and the right hand side goes to 1, there
		will always exist $\gamma_k\in I_k$ fulfilling both criteria. Finally, by
		starting from a point $(x_0, \alpha_0)$ satisfying the Tikhonov normal equations,
		we know that the inverse Jacobian exists in the first iteration.
\end{proof}

From this theorem it follows that as long as $\gamma_k$ is chosen small enough,
the Newton iterations will converge. Small values will however lead to slow 
convergence, so we will derive an upper bound for $\gamma_k$. In order to do this
we will simplify the dependency of the upper bound for $\left\|E_k\right\|$ found
in \hypref{lemma}{lem:normEandM} on $\sqrt{1 + \zeta_k^2}$.
\begin{lemma}\label{lem:theta}
		For all $\omega\in]0, 1[$ the following holds:
		\[
				\sqrt{1 + \zeta^2}\leq
				\left\{\begin{aligned}
						&\sqrt{2} && \text{If }\Delta\alpha_k > 0\\
						&\sqrt{1 + \left(\frac{\alpha_{k - 1}}{\alpha_{k - 1} + \Delta\alpha_k}\right)^2}
								&& \text{If }\Delta\alpha_k < 0\text{ and }\alpha_{k - 1} + \Delta\alpha_k > 0\\
						&\sqrt{1 + \frac{1}{\left(1 - \omega\right)^2}} && \text{If } \Delta\alpha_k < 0\text{ and }
								\alpha_{k - 1} + \Delta\alpha_k < 0
				\end{aligned}\right.
		\]
\end{lemma}
\begin{proof}
		Finding an upper bound for
		\[
				\sqrt{1 + \zeta^2} = \sqrt{1 + \left(\frac{\alpha_{k - 1}}{\alpha_{k - 1} +
						\gamma_k\Delta\alpha_k}\right)^2}
		\]
		is equivalent to finding a lower bound on $\left|\alpha_{k - 1} +
		\gamma_k\Delta\alpha_k\right|$.
		\begin{itemize}
				\item If $\Delta\alpha_k > 0$, then $I_k = ]0, 1]$ and this lower bound
						is found for $\gamma_k = 0$.
				\item If $\Delta\alpha_k < 0$ and $\alpha_{k - 1} + \Delta\alpha_k > 0$
						(meaning that using the unscaled Newton iteration would give a positive regularization parameter),
						then $I_k = ]0, 1]$ and this lower bound is found for $\gamma_k = 1$.
				\item If $\Delta\alpha_k < 0$ and $\alpha_{k - 1} + \Delta\alpha_k < 0$
						(meaning that using the unscaled Newton iteration would give a negative regularization parameter),
						then $I_k = \left]0, -\omega\alpha_{k - 1}/\Delta\alpha_k\right]$. If
						$\omega\rightarrow 1$ then $\alpha_{k - 1} + \gamma_k\Delta\alpha_k
						\rightarrow 0$ and $\sqrt{1 + \zeta^2}\rightarrow +\infty$. In order to
						avoid this we take $\omega\in]0, 1[$ to stay way from this singularity
						and find the lower bound for $\gamma_k = -\omega\alpha_{k - 1}/\Delta\alpha_k$.
		\end{itemize}
		Substituting these values for $\gamma_k$ proves the lemma.
\end{proof}
\begin{corollary}\label{cor:gamma1}
		If $\theta_k$ is the bound on $\sqrt{1 + \zeta_k^2}$ from \hypref{lemma}
		{lem:theta}, then the following step size fulfils the conditions
		\eqref{eq:constr1} and \eqref{eq:constr2} of \hypref{theorem}{thm:gamma}:
		\[
				\gamma_k = \min\left\{\max I_k, \frac{1}{\left(\sqrt{\Delta\alpha_k^2 + \frac{1}{4}
						\left\|A^TA\Delta x_k\right\|^2} + \left|\Delta\alpha_k\right| +
						\theta_k\left\|\Delta x_k\right\|\right)\left\|D_{k - 1}^{-1}\right\|}
						\right\}
		\]
\end{corollary}
\begin{proof}
		This result is found by replacing $\left\|E_k\right\|$, $\left\|M_k\right\|$ and
		$\sqrt{1 + \zeta^2}$ in \eqref{eq:constr2} by their upperbounds found in \hypref{lemmas}
		{lem:normEandM} and \ref{lem:theta}.
\end{proof}
\begin{corollary}\label{cor:gamma2}
		If $\theta_k$ is the bound on $\sqrt{1 + \zeta_k^2}$ from \hypref{lemma}
		{lem:theta}, then the following step size only fulfils conditions \eqref{eq:constr1}
		of \hypref{theorem}{thm:gamma}:
		\[
				\gamma_k = \min\left\{\max I_k, \frac{1}{\left(\left|\Delta\alpha_k\right|
						+ \theta_k\left\|\Delta x_k\right\|\right)\left\|D_{k - 1}^{-1}\right\|}
						\right\}
		\]				
\end{corollary}
\begin{proof}
		This result is found by replacing $\left\|E_k\right\|$ and $\sqrt{1 + \zeta^2}$
		in \eqref{eq:constr1} by their upperbounds found in \hypref{lemmas}{lem:normEandM}
		and \ref{lem:theta}.
\end{proof}
\noindent Combining the results from this section leads to \hypref{algorithm}{alg:ntm}.

\begin{algorithm}
		\caption{Newton on the Tikhonov-Morozov system (NTM)}\label{alg:ntm}
		\begin{algorithmic}[1]
				\State Choose initial $\alpha_0 > 0$ and solve $F_1(x_0, \alpha_0)$ for $x_0$.
				\For{$k = 1, \ldots,$ maxiter}
						\State Solve the Jacobian system \eqref{eq:newtoneq} for $\Delta x_k$ and $\Delta\alpha_k$.\label{alg:ntm:jac}
						\State Calculate $\left\|D^{-1}\right\|$.
						\State Calculate $\theta_k$ using \hypref{lemma}{lem:theta}.
						\State Calculate the step size $\gamma_k$ using \hypref{corollary}{cor:gamma1} or \ref{cor:gamma2}.
						\State $x_k = x_{k - 1} + \gamma_k\Delta x_k$ and $\alpha_k = \alpha_{k - 1} + \gamma_k\Delta\alpha_k$.
						\If{$\left\|F(x_k, \alpha_k)\right\| <$ tol}
										\Break
						\EndIf
				\EndFor
		\end{algorithmic}
\end{algorithm}


\subsection{Remarks}
The reason we consider two possible choices for the step size is because we 
observed in our numerical experiments that both \hypref{corollary}{cor:gamma1}
and \ref{cor:gamma2} seem to result in a small value for the step size. This is
explained by the fact that the constraints placed on $\gamma_k$ in \hypref{theorem}
{thm:gamma} are stronger than \hyprefp{C}{c1} and \hyprefp{C}{c2} and because
we used various overestimations in order to derive an upper bound for $\gamma_k$.

Another thing to note is that it might not be necessary to start from a point
$(x_0, \alpha_0)$ on the discrepancy curve. We use this assumption because it
guarantees the existence of the inverse Jacobian in the first iteration. However,
as \hypref{theorem}{thm:tref} suggests, it would be sufficient to start from a point
for which the perturbation $E$ in the Jacobian $J$ with respect to $D$ sufficiently small.
Instead of choosing an $\alpha_0$ and solving $F_1(x_0, \alpha_0) = 0$ for $x_0$ exactly,
it could suffice to only solve for $x_0$ up to a limited precision.

Finally, for large scale problems, solving the Jacobian system \eqref{eq:newtoneq}
and calculating $\left\|D_{k - 1}^{-1}\right\|$ becomes computationally very
expensive. We could use the upper bound from \hypref{lemma}{lem:normDinv} to partially
solve this problem, but once again, this will only lead to a smaller step size and
slower convergence. These issues will be discussed further on in this paper.


\section{Numerical experiments I}\label{sec:numexp1}
To illustrate the method, we look at a problem with a small random matrix $A$
and solution $x$. More precisely, we take $A\in\mbbR^{700\times 500}$ and $x\in\mbbR^{500}$
with i.i.d. entries drawn from the uniform distribution $\mc{U}(-1, 1)$. Measurements
are generated by adding $10\%$ Gaussian noise to the exact right hand side $b_{ex} = Ax$
using $e\sim\mathcal{N}\left(0, \sigma^2 I_m\right)$ with $\sigma = 0.10\left\|
b_{ex}\right\|/\sqrt{m}$ and setting $b = b_{ex} + e$. For the discrepancy principle,
we will approximate the error norm by $\varepsilon = \sigma\sqrt{m} = 0.10\left\|b_{ex}\right\|$.

We repeat this experiment $1000$ times and for each run we start with
$\alpha_0 = 1$ and solve the Tikhonov normal equations $F_1(x_0, \alpha_0) = 0$
for $x_0$. After that, we start the Newton iterations with $\omega = 0.9$ and
stop when $\left\|F(x_k, \alpha_k)\right\| < 1\snot{-3}$. The results are show in
\hypref{figure}{fig:rms} and \hypref{table}{tab:rms}, where case 1
means that \hypref{corollary}{cor:gamma1} was used to calcuate the step
size and case 2 means that \hypref{corollary}{cor:gamma2} was used.

These results indicate that the overestimations used in our analysis
of the method lead to a small step size. By using
\hypref{corollary}{cor:gamma2} and weakening the constraints placed on $\gamma$,
the method takes substantially larger steps and converges much faster.
How much larger the step sizes can become by
weakening the constraints is of course problem dependent and hard to predict.
Nevertheless, \hyprefp{C}{c2} seems to be a strong constraint placed on the iterations. Also, because both cases converge to the same
solution, the same regularization parameter is found. The small standard deviation
over all the runs indicates that the regularization parameter is quite
similar in all the runs.

\begin{figure}
		\centering
		\includegraphics[width = 0.49\linewidth]{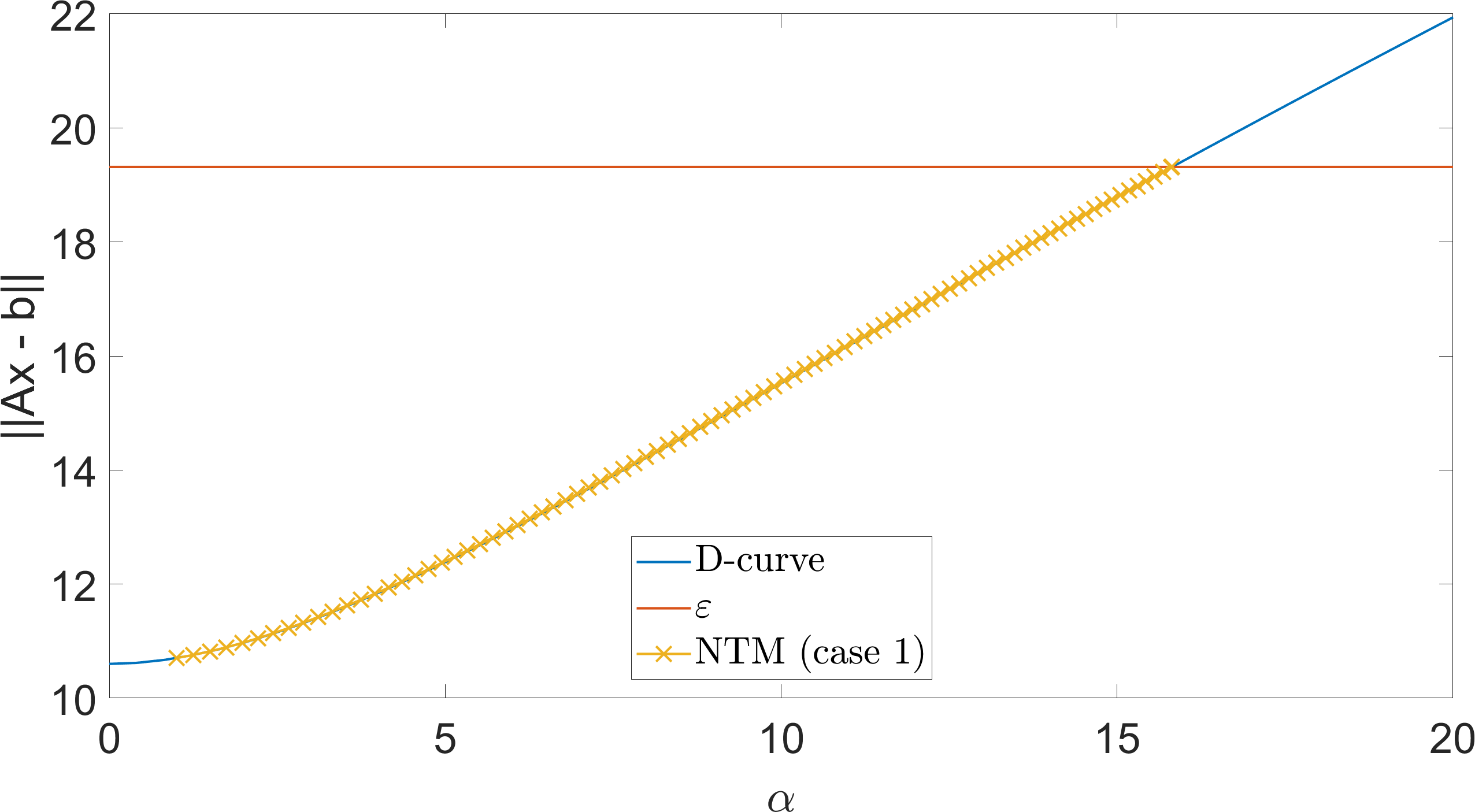}\hspace{2.5pt}
		\includegraphics[width = 0.49\linewidth]{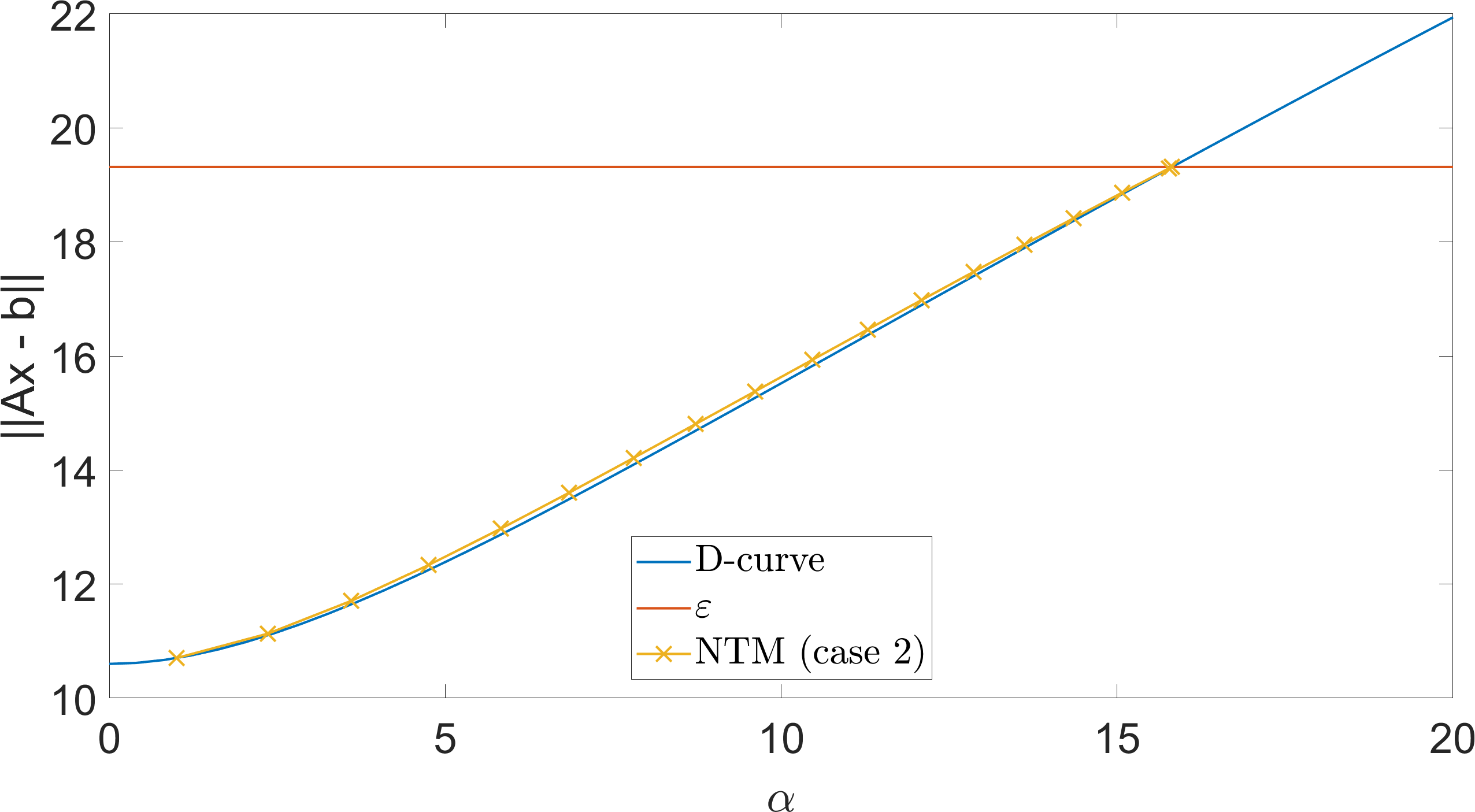}\\[2.5pt]
		\includegraphics[width = 0.49\linewidth]{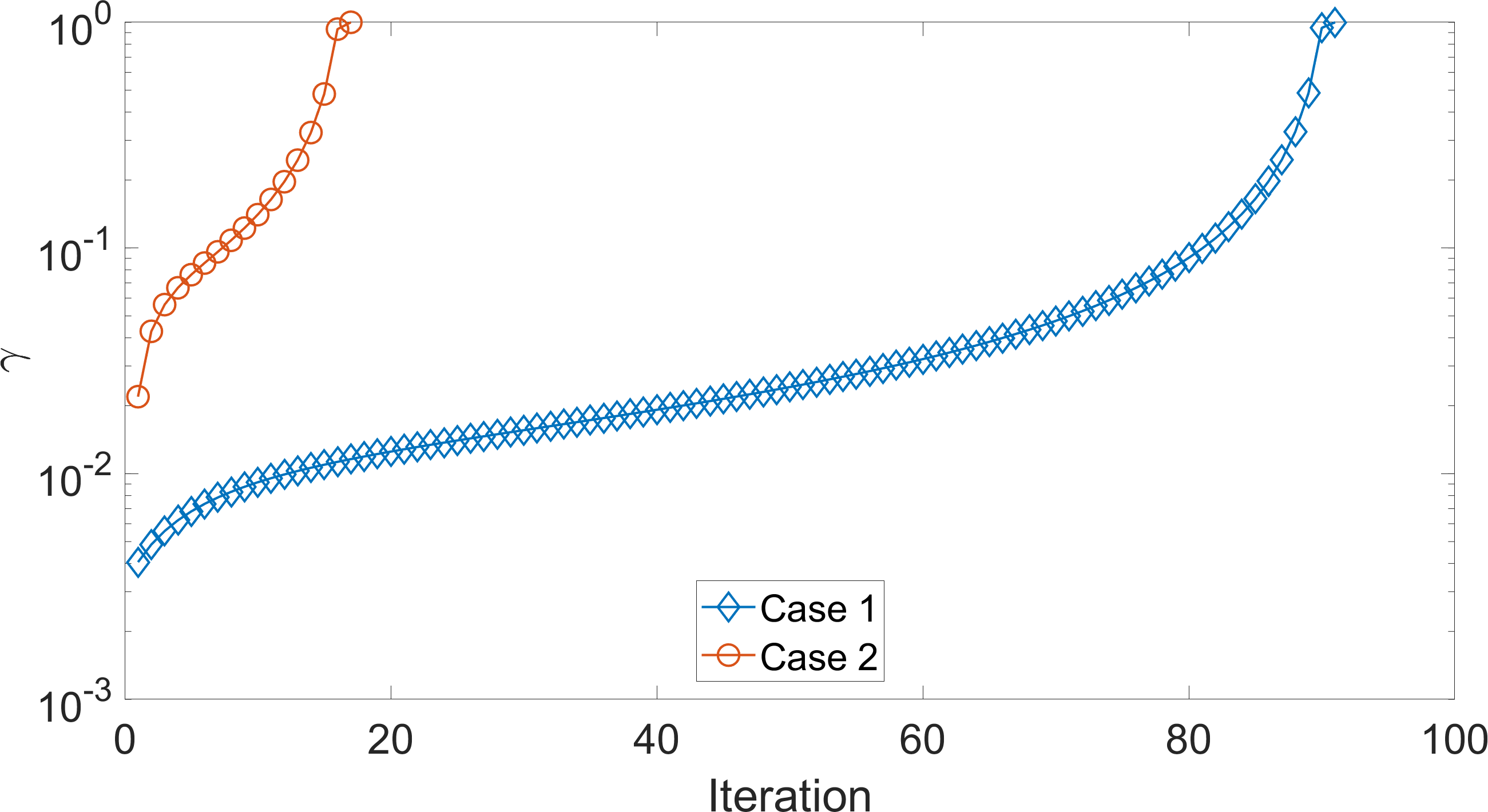}\hspace{2.5pt}
		\includegraphics[width = 0.49\linewidth]{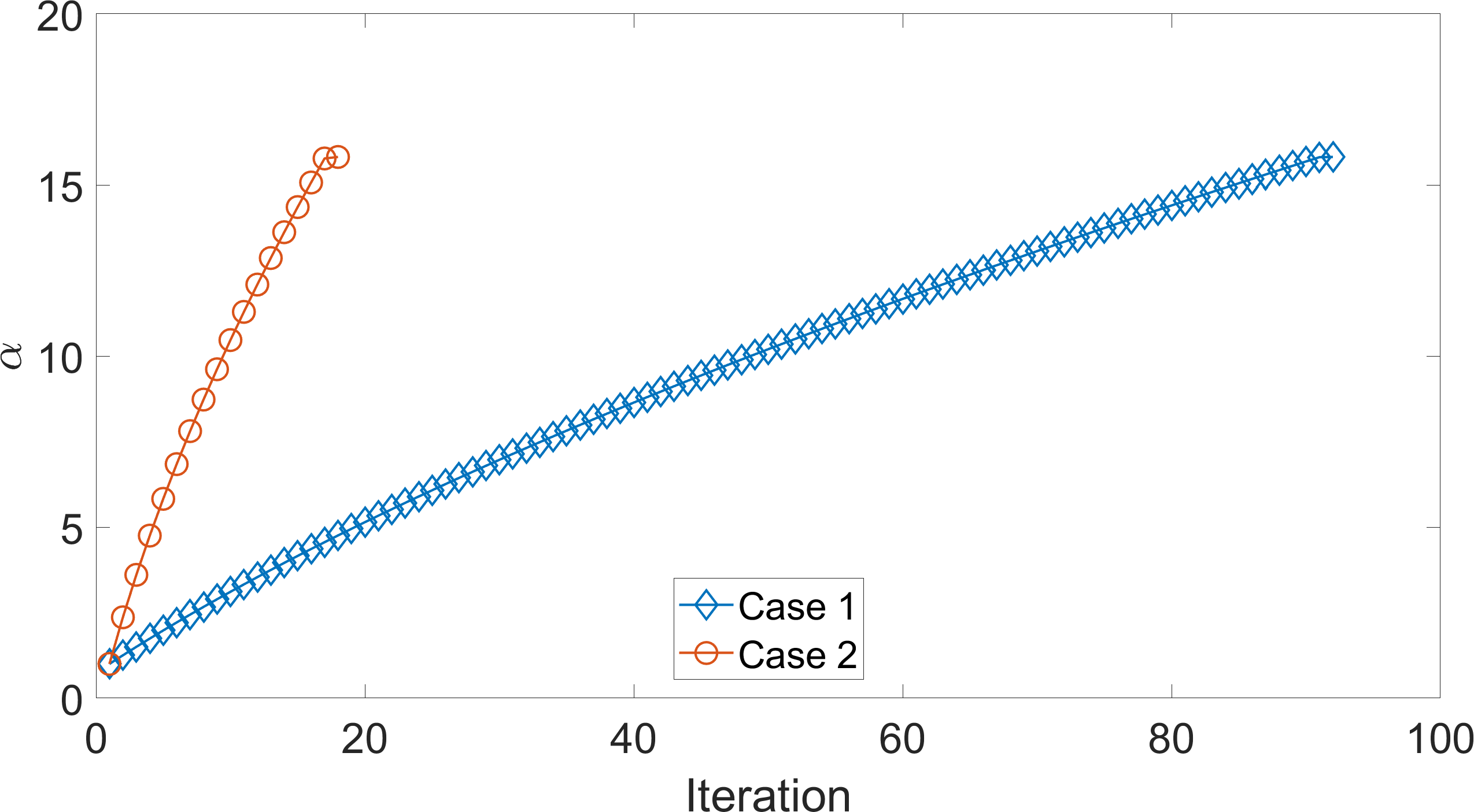}
		\caption{Results for one of the runs. For each Newton iteration, we plot the
				point	$(\alpha_k, \left\|Ax_k - b\right\|)$. The top left figure
				corresponds to case 1 and the top right figure to case 2. Bottom left:
				the value of the step size $\gamma$ used in each iteration. Bottom right:
				the value of the regularization parameter $\alpha$ in each iteration.}
		\label{fig:rms}
\end{figure}

\begin{table}
		\centering
		\begin{tabular}{c||c|c}
				& \# Iterations & $\alpha$\\ \hline\hline
				Case 1 & $85$ ($13$) & $15.6581$ ($1.0947$) \\ \hline
				Case 2 & $16$ ($2$)  & $15.6581$ ($1.0947$)
		\end{tabular}
		\caption{Average number of iterations for the 1000 runs of the experiment and
				the standard deviation (rounded). Because both methods converge to the same
				solution, the same value for $\alpha$ is found in each run, but for all the
				different random matrices its value turns out to be quite similar, hence
				the low standard deviation.}
		\label{tab:rms}
\end{table}


\section{Projected Tikhonov-Morozov system}\label{sec:pntm}
The NTM algorithm can become computationally very expensive because in each
iteration $\left\|D^{-1}\right\|$ needs to be computed and the Jacobian system
\eqref{eq:newtoneq} needs to be solved for $\Delta x$ and $\Delta\alpha$. Even
for small matrices $A\in\mbbR^{m\times n}$ this can quickly become a problem.
However, it is possible to project the problem onto a Krylov subspace \cite{saad2003,
vandervorst2003} using a bidiagonal decomposition of $A$ \cite{golub1965, paige1982,
paige1982_2}. In each outer Krylov iteration, the projected version of the
Tikhonov-Morozov system \eqref{eq:tikmor} can then be solved using the NTM algorithm.

In this section we describe how this algorithm works using a number of heuristic
choices and apply it to different test problems. Roughly speaking, each iteration of the method will consist of the following steps:
\begin{itemize}
		\item Expand the bidiagonal decomposition of $A$.
		\item Choose an initial point for the NTM method on the projected equations.
		\item Calculate a number of NTM iterations on the projected equations.
		\item Check the convergence.
\end{itemize}


\subsection{Bidiagonal decomposition}
\begin{theorem}[Bidiagonal decomposition]\label{thm:bidiag}
		If $A\in\mbbR^{m\times n}$ with $m\geq n$, then there exist orthonormal matrices
		\[
				U = (u_1, u_2, \ldots, u_m)\in\mbbR^{m\times m}\quad\text{and}\quad
						V = (v_1, v_2, \ldots, v_n)\in\mbbR^{n\times n}
		\]
		and a lower bidiagonal matrix
		\[
				B=\begin{pmatrix}\mu_1\\ \nu_2 & \mu_2\\  & \nu_3 & \ddots\\  &  &
						\ddots & \mu_{n}\\ 
						&  &  & \nu_{n+1}\end{pmatrix}\in\mathbb{R}^{(n+1)\times n},
		\]
		such that
		\[
				A = U\begin{pmatrix}B\\0\end{pmatrix}V^T.
		\]
\end{theorem}
\begin{proof}
		This was proven by Golub and Kahan in \cite{golub1965}.
\end{proof}

Starting from a given unit vector $u_1\in\mathbb{R}^m$ it is possible to
generate the columns of $U$, $V$ and $B$ recursively using the Bidiag1 procedure
proposed by Paige and Saunders \cite{paige1982_2, paige1982}, see \hypref{algorithm}
{alg:bidiag1}. Here, the reorthogonalization is added for numerical stability.
Note that this bidiagonal decomposition is the basis for the LSQR algorithm and
that after $k$ steps of Bidiag1 starting with the initial vector $u_1 = b/\left\|b\right\|$
we have matrices $V_k\in\mbbR^{n\times k}$ and $U_{k + 1}\in\mbbR^{m\times(k + 1)}$
with orthonormal columns and a lower bidiagonal matrix $B_{k + 1, k}\in\mbbR^{(k + 1)\times k}$
that satisfy
\begin{equation}\label{eq:bdrel}
		AV_k = U_{k + 1}B_{k + 1, k}
\end{equation}

\begin{algorithm}
		\caption{bidiag1}\label{alg:bidiag1}
		\begin{algorithmic}[1]
				\State Choose initial unit vector $u_1$ (typically $b/\left\|b\right\|$).
				\State Set $\nu_1v_0 = \mu_{n + 1}v_{n = 1} = 0$.
				\For{$k = 1, \ldots,$ n}
						\State $r_k = A^Tu_k - \nu_kv_{k-1}$
						\State Reorthogonalize $r_k$ with respect to the previous columns of $V$.
						\State $\mu_k = \left\|r_k\right\|$ and $v_k = r_k/\mu_k$.
						\State $p_k = Av_k - \mu_ku_k$
						\State Reorthogonalize $p_k$ with respect to the previous columns of $U$.
						\State $\nu_{k+1} = \left\|p_k\right\|$ and $u_{k+1} = p_k/\nu_{k+1}$.
				\EndFor
		\end{algorithmic}
\end{algorithm}

In order to solve the Tikhonov-Morozov system \eqref{eq:tikmor}, we will calculate
a series of iterations in the Krylov subspace spanned by the columns of $V$:
\[
		x_k\in\spn{V_k} = \mc{K}_k(A^TA, A^Tb).
\]
This means that $x_k = V_ky_k$ for some $y_k\in\mbbR^k$ and using \eqref{eq:bdrel},
the orthonormality of the columns of $U$ and $V$ and the fact that $u_1 = b/\left\|
b\right\|$ it is possible to show that
\begin{equation}\label{eq:ptik}
		\min_{x_k\in\spn{V_k}}\left\|Ax_k - b\right\|^2 + \alpha\left\|x_k\right\|^2 =
				\min_{y_k\in\mbbR^n}\left\|B_{k + 1, k}y_k - c_k\right\|^2 + \alpha\left\|y_k\right\|^2
\end{equation}
and
\begin{equation}\label{eq:pmor}
		\left\|Ax_k - b\right\| = \left\|B_{k + 1, k}y_k - c_k\right\|,
\end{equation}
for $c_k = (\left\|b\right\|, 0, \ldots, 0)^T\in\mbbR^{k + 1}$. We therefore
set $x_k = V_ky_k$ and solve the following projected version of \eqref{eq:tikmor}:
\begin{equation}\label{eq:ptikmor}
		\begin{aligned}		
				\left\{\begin{aligned}
						\wt{F}_1(y_k, \alpha_k) &= (B_{k + 1, k}^TB_{k + 1, k} + \alpha I_k)y_k -
								B_{k + 1, k}^Tc_k\\
						\wt{F}_2(xy_k \alpha_k) &= \frac{1}{2}\left(B_{k + 1, k}y_k - c\right)^T
								\left(B_{k + 1, k}y_k - c_k\right) - \frac{1}{2}\varepsilon^2
				\end{aligned}\right.
		\end{aligned}
\end{equation}
Similarly to the to original non-linear system, $\wt{F}_1$ are the normal equations
corresponding to the projected Tikhonov problem \eqref{eq:ptik} and $\wt{F}_2$
corresponds to the projected discrepancy principle \eqref{eq:pmor}.


\subsection{Inner NTM iterations}
In each outer Krylov iteration (numbered with $k$) \eqref{eq:ptikmor} needs to
be solved, which we will do using the NTM method. This means that in the inner
Newton iterations (numbered with $l$), the following Jacobian system needs to be
solved:
\begin{equation}\label{eq:pjacsys}
		\begin{aligned}
				&\begin{pmatrix} B_{k + 1, k}^TB_{k + 1, k} + \alpha_{k, l - 1}I_k & y_{k, l - 1}\\
						\frac{1}{\alpha_{k, l - 1}}\left(B_{k + 1, k}y_{k, l - 1} - c_k\right)^TB_{k + 1, k} & 
						0\end{pmatrix}\begin{pmatrix}\Delta y_{k, l} \\ \Delta\alpha_{k, l}\end{pmatrix}\\
				&\qquad\quad=-\begin{pmatrix} \left(B_{k + 1, k}^TB_{k + 1, k} + \alpha_{k, l - 1}I_k\right)y_{k, l - 1}
						- B_{k + 1, k}^Tc_k \\ \frac{1}{2\alpha_{k, l - 1}}\left(B_{k + 1, k}y_{k, l - 1} - c_k\right)^T
						\left(B_{k + 1, k}y_{k, l - 1} - c_k\right) - \frac{1}{2\alpha_{k, l - 1}}\varepsilon^2\end{pmatrix}.
		\end{aligned}
\end{equation}
Note that the matrix $B_{k + 1, k}^TB_{k + 1, k}$ has size $k\times k$. This
means that as long as the number of outer iterations remains small -- which
corresponds to the size of the constructed Krylov basis -- calculating $\left\|
D_{l - 1}^{-1}\right\|$ and solving the projected Jacobian system \eqref{eq:pjacsys}
of size $(k + 1)\times(k + 1)$ can be done efficiently. A full overview of the
method can be found in \hypref{algorithm}{alg:pntm} and below we discuss some
of the steps.

As a starting point for the original NTM method, we used the solution to the
Tikhonov normal equations $F_1(x_0, \alpha_0) = 0$ for a chosen $\alpha_0$. Now,
in each outer Krylov iteration, we will use the current best estimate for the
regularization parameter, i.e. $\alpha_{k, 0} = \alpha_{k - 1}$ and solve the
projected Tikhonov normal equations $\wt{F}_1(y_{k, 0}, \alpha_{k, 0})$. This
$k \times k$ linear system can be solved quickly as long as the number of Krylov
iterations is small and its solution can be used to initialize the inner Newton
iterations.

Another important question is how many inner Newton iterations should be performed
before the Krylov subspace is expanded. If, on the one hand, the Krylov subspace
is too small to contain the solution $x$ of the inverse problem (or a good approximation
of it), then the Newton iterations cannot converge. Therefore we would like the
number of inner iterations to be small. If, on the other hand, the Krylov subspace
is large enough to contain the solution, we don't want to keep expanding it. The
maximum number of inner Newton iterations should therefore be large enough for
them to converge. This is why we initially limit the number of inner Newton
iterations. However, the moment that the residual
of the solution becomes less than the discrepancy level $\varepsilon$, we will take
a much larger number. This corresponds to \hypref{lines}{alg:pntm:init1}--\ref{alg:pntm:init2}
of \hypref{algorithm}{alg:pntm}.

Finally, we don't change the stopping criterion for the inner Newton iterations,
\hypref{algorithm}{alg:pntm} \hypref{line}{alg:pntm:stop1}. However, because
we are now working with the the projected system, $\wt{F}$ may be solved accurately
before the original system $F$ is. We therefore don't stop the outer Krylov iterations
until the value for the regularization parameter $\alpha_{k}$ stagnates as well,
\hypref{algorithm}{alg:pntm} \hypref{line}{alg:pntm:stop2}. The necessity for this
will become clear in the numerical experiments, where we will see that this corresponds
to finding a solution $x_k$ that satisfied the discrepancy principle, but not the Tikhonov
normal equations.

\begin{algorithm}
		\caption{Projected Newton on the Tikhonov-Morozov system (PNTM)}\label{alg:pntm}
		\begin{algorithmic}[1]
				\State Choose initial $\alpha_0 > 0$.
				\State Set FLAG $= 0$.
				\For{$k = 1, \ldots,$ outeriter}
						\State Expand $U_{k + 1}$, $B_{k + 1, k}$ and $V_k$ using Bidiag1 \eqref{alg:bidiag1}.
						\State Set $\alpha_{k, 0} = \alpha_{k - 1}$ and solve $\wt{F}_1(y_{k, 0}, \alpha_{k, 0})$ for $y_{k, 0}$.
						\If{$\left\|B_{k + 1, k}y_{k, 0} - c_k\right\| > \epsilon$}\label{alg:pntm:init1}
								\State inneriter $= \min\left\{k, 10\right\}$
						\Else
								\State inneriter $= 10000$\label{alg:pntm:inneriter}
						\EndIf\label{alg:pntm:init2}
						\For{$l = 1, \ldots,$ inneriter}
								\State Solve the projected Jacobian system \eqref{eq:pjacsys} for $\Delta y_{k, l}$ and $\Delta\alpha_{k, l}$.
								\State Calculate $\left\|D^{-1}_{l - 1}\right\|$.
								\State Calculate the step size $\gamma$ using \hypref{corollary}{cor:gamma1} or \ref{cor:gamma2}.
								\State $y_{k, l} = y_{k, l - 1} + \gamma\Delta y_{k, l}$ and $\alpha_{k, l} = \alpha_{k, l - 1} + \gamma\Delta\alpha_{k, l}$.
								\If{$\left\|\wt{F}(y_{k, l}, \alpha_{k, l})\right\| <$ tol}\label{alg:pntm:stop1}
										\State FLAG $= 1$
										\Break
								\EndIf
						\EndFor
						\State $x_k = V_ky_{k, l}$ and $\alpha_k = \alpha_{k, l}$.
						\If{FLAG $= 1$ \textbf{and} $\left|\alpha_k - \alpha_{k - 1}\right|/\alpha_{k - 1} <$ tol}\label{alg:pntm:stop2}
								\Break
						\EndIf
				\EndFor
		\end{algorithmic}
\end{algorithm}


\section{Reference methods}\label{sec:refmethods}
In this section we briefly discuss two methods which we compare the PNTM method to.
The first method iteratively solves the Tikhonov problem and also uses an iterative
update scheme for the regularization parameter based on the discrepancy principle.
The second method does not solve the Tikhonov problem, but combines an early stopping
criterion with a right preconditioner in order to include prior knowledge and
regularization.


\subsection{Generalized bidiagonal-Tikhonov}
In \cite{gazzola2014_1, gazzola2014_2, gazzola2014_3} a generalized Arnoldi-Tikhonov
method (GAT) was introduced that iteratively solves the Tikhonov problem \eqref{eq:tikhonov}
using a Krylov subspace method based on the Arnoldi decomposition of the matrix $A$.
Simultaneously, after each Krylov iteration, the regularization parameter is updated
in order to approximate the value for which the discrepancy is equal to $\varepsilon$.
This is done using one step of the secant method to find the intersection of the
discrepancy curve with the tolerance for the discrepancy principle, see \hypref{figure}
{fig:curves}, but in the current Krylov subspace. Because the method is based on
the Arnoldi decomposition, the method is connected to the GMRES algorithm and it
only works for square matrices. However, by replacing the Arnoldi decomposition
with the bidiagonal decomposition we used in the previous section the method can
be adapted to non-square matrices.

The update for the regularization parameter is done based on the regularized and
the non-regularized residual. Let, in the $k$th iteration, $z_k$ be the solution
without regularization -- i.e. $\alpha = 0$ -- and $y_k$ the solution with the
current best regularization parameter -- i.e. $\alpha = \alpha_{k - 1}$. If $r(z_k)$
and $r(y_k)$ are the corresponding residuals, then the regularization parameter
is updates using
\begin{equation}\label{eq:aup}
		\alpha_k = \left|\frac{\varepsilon - r(z_k)}{r(y_k) - r(z_k)}\right|\alpha_{k - 1}.
\end{equation}
A brief sketch of this method is given is \hypref{algorithm}{alg:gbit}, where
we use the same stopping criterion as for PNTM, but for more information we
refer to \cite{gazzola2014_1, gazzola2014_2, gazzola2014_3}. Note that in the
original GAT method, the non-regularized iterates $z_k$ are equivalent to the
GMRES iterations for the solution of $Ax = b$. Now, because the Arnoldi
decomposition is replaced with the bidiagonal decomposition, they are equivalent
to the LSQR iterations for the solution of $Ax = b$.

\begin{algorithm}
		\caption{Generalized bidiagonal Tikhonv (GBiT)}\label{alg:gbit}
		\begin{algorithmic}[1]
				\State Choose initial $\alpha_0 > 0$.
				\For{k = 1, \ldots, maxiter}
						\State Expand $U_{k + 1}$, $B_{k + 1, k}$ and $V_k$ using Bidiag1 \eqref{alg:bidiag1}.
						\State Solve $\wt{F}_1(z_k, 0) = 0$ for $z_z$.
						\State Solve $\wt{F}_1(y_k, \alpha_{k - 1}) = 0$ for $y_k$.
						\State Calculate $\alpha_k$ using \eqref{eq:aup}.
						\If{$\left\|\wt{F}(y_k, \alpha_k)\right\| <$ tol \textbf{and} $\left|\alpha_k - \alpha_{k - 1}\right|/\alpha_{k - 1} <$ tol}
								\Break
						\EndIf
				\EndFor
		\end{algorithmic}
\end{algorithm}


\subsection{General form Tikhonov and priorconditioning}
In its general form, the Tikhonov problem \eqref{eq:tikhonov} is written as
\begin{equation}\label{eq:gentikhonov}
		x_\alpha = \argmin_{x\in\mbbR^n}\left\|Ax - b\right\|^2 + \alpha\left\|L(x - x_0)\right\|^2,
\end{equation}
with $x_0\in\mbbR^n$ an initial estimate and $L\in\mbbR^{p\times n}$ a regularization
matrix, both chosen to incorporate prior knowledge or to place specific constraints
on the solution \cite{gazzola2014_1, hansen2010}. If $L$ is a square invertible matrix, then the problem can
be written in the standard form 
\begin{equation}\label{eq:tranftikhonov}
			z_\alpha = \argmin{z\in\mbbR^n}\left\|\ol{A}z - r_0\right\|^2 + \alpha\left\|z\right\|^2,
\end{equation}
by using the transformation
\begin{equation}\label{eq:stdtransf}
		z = L(x - x_0), \quad\ol{A} = AL^{-1}, \quad r_0 = b - Ax_0.
\end{equation}
When $L$ is not square invertible, some form of pseudoinverse has to be used,
but the reformulation of the problem remains the same \cite{hansen2010}.

After solving \eqref{eq:tranftikhonov}, the solution can be found as
\[
		x = x_0 + L^{-1}z.
\]
Instead of solving $Ax = b$, an alternative regularization method called priorconditionning
is to solve
\[
		\left\{\begin{aligned}
				&AL^{-1}z = b - Ax_0\\
				&x = x_0 + L^{-1}z
		\end{aligned}\right.
\]
Here, the matrix $L$ is can be seen as a right preconditioner. Its functions is,
however, not to improve the convergence of the iterative method, but to incorporate
regularization and prior knowledge into the solution \cite{calvetti2015}. This priorconditionned
linear system can now be solved with CGLS combined with an early stopping criterion
based on the discrepancy principle. Note that this method will find a solution in the
same Krylov subspace as PNTM, but that PNTM selects another element of this space
due to the presence of the regularization term.


\section{Numerical Experiments II}\label{sec:numexp2}


\subsection{Large random matrix problem}
As a first numerical experiment, we repeat the random matrix experiment from
\hypref{section}{sec:numexp1}. The only thing we change is
the size of the matrices: $21000\times 15000$. The results are shown in \hypref{figure}
{fig:rml} and \hypref{table}{tab:rml}, where we used $tol = 1\snot{-3}$ for the 
stopping criterion. Similarly as with the smaller experiment,
there is little difference between the different runs when it comes to the number
of iterations (outer and inner) or the optimal regularization parameter. As a
comparison, we also solved the problem with GBiT and see that while a similar
value for the regularization parameter is found, PNTM requires less Krylov
iterations in order to converge.

When we compare \hypref{figure}{fig:rms} and \hypref{figure}{fig:rml}, we see
that the behaviour of the method is quite different now. In the original NTM
method we started from a point on the discrepancy curve and stayed close to it
by limiting the step size. Now, with the PNTM method, we solve the problem in
Krylov subspaces of increasing size. This means that in the first few iterations,
we end up far away from the true discrepancy curve. At some point we have constructed
a Krylov subspace in which we can solve the projected system up to the discrepancy
principle, but as we observe, not necessarily the true Tikhonov normal equations.
At this point we increase the maximum number of inner iterations and we keep
performing outer Krylov iterations until the regularization parameter stagnates.

Whichever of the two corollaries we use to determine the step size produces similar results.
The main difference is the number of inner iterations required to solve the projected
system. Using \hypref{corollary}{cor:gamma2}, the method once again requires a significantly
lower number of Newton iterations to converge inside each of the Krylov subspaces.

\begin{figure}
		\centering
		\includegraphics[width = 0.49\linewidth]{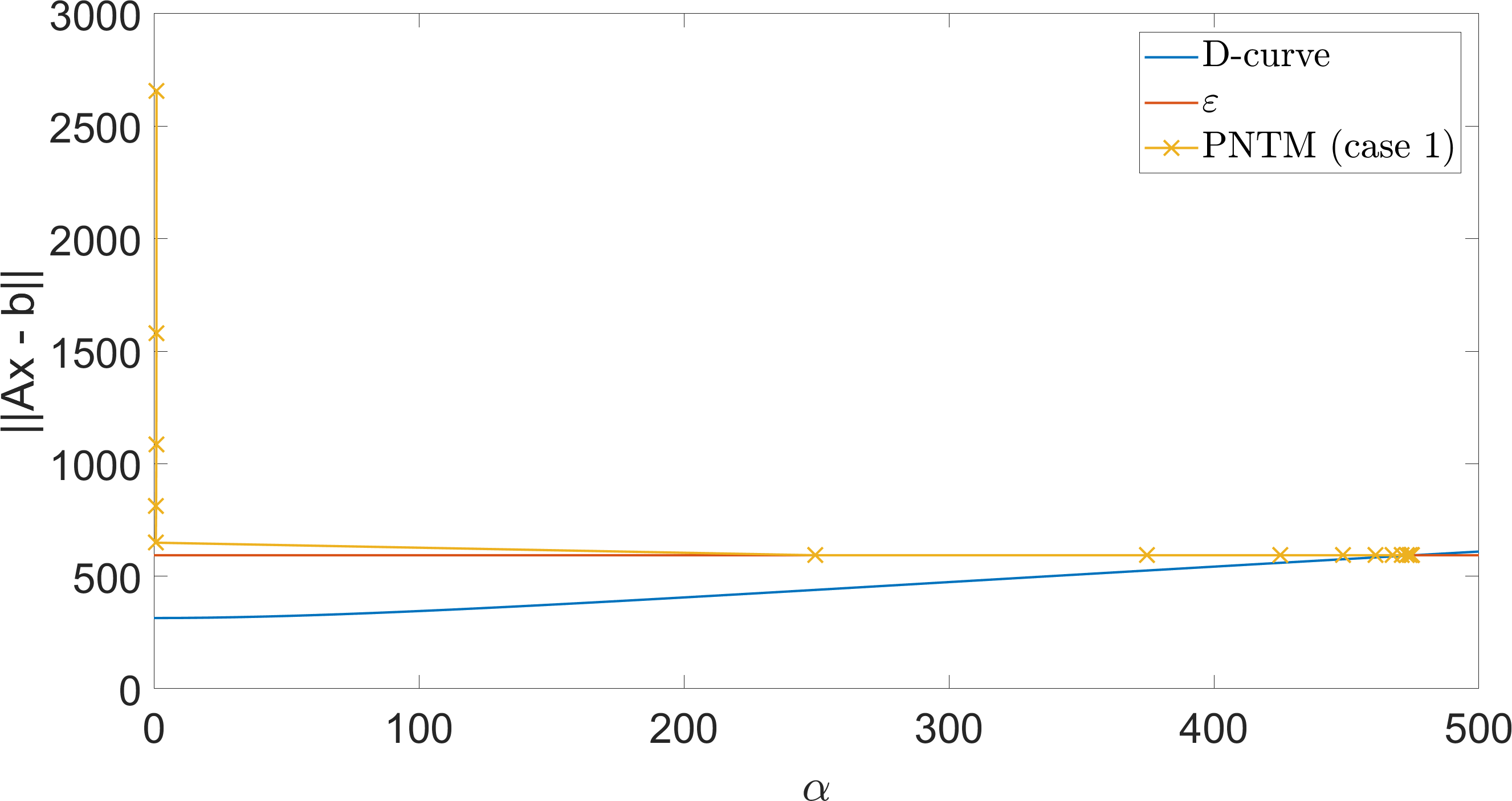}\hspace{2.5pt}
		\includegraphics[width = 0.49\linewidth]{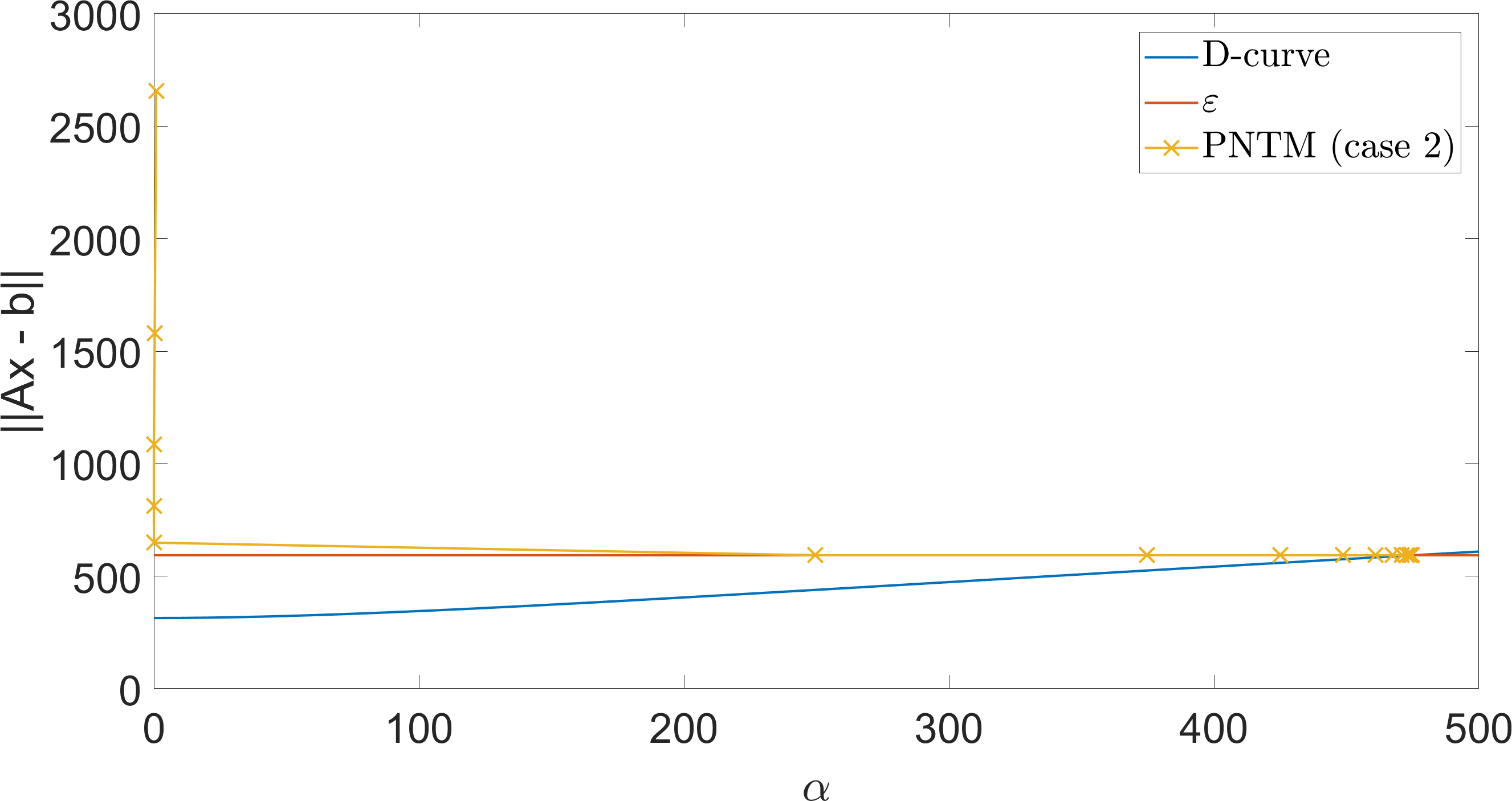}\\[2.5pt]
		\includegraphics[width = 0.49\linewidth]{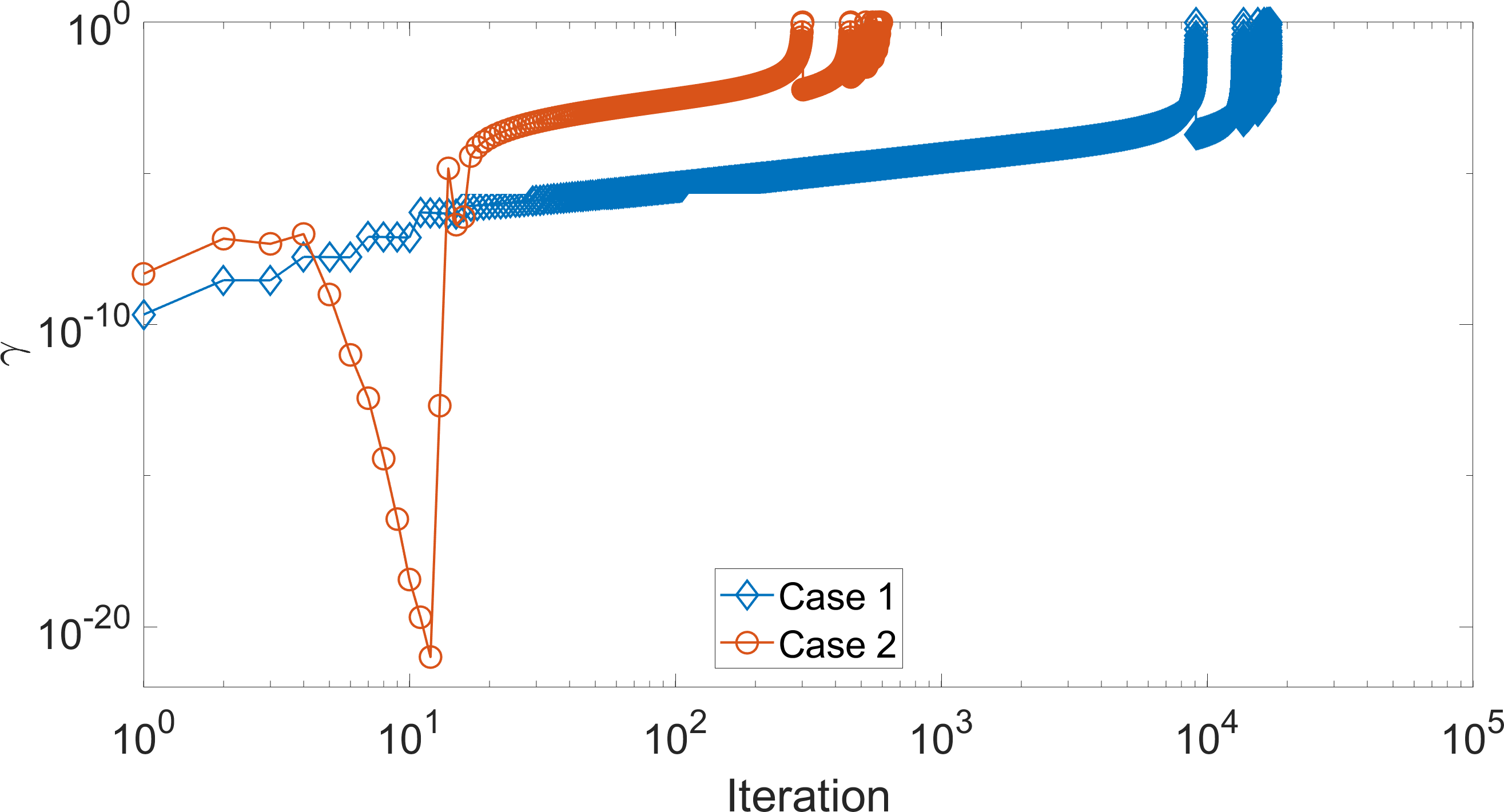}\hspace{2.5pt}
		\includegraphics[width = 0.49\linewidth]{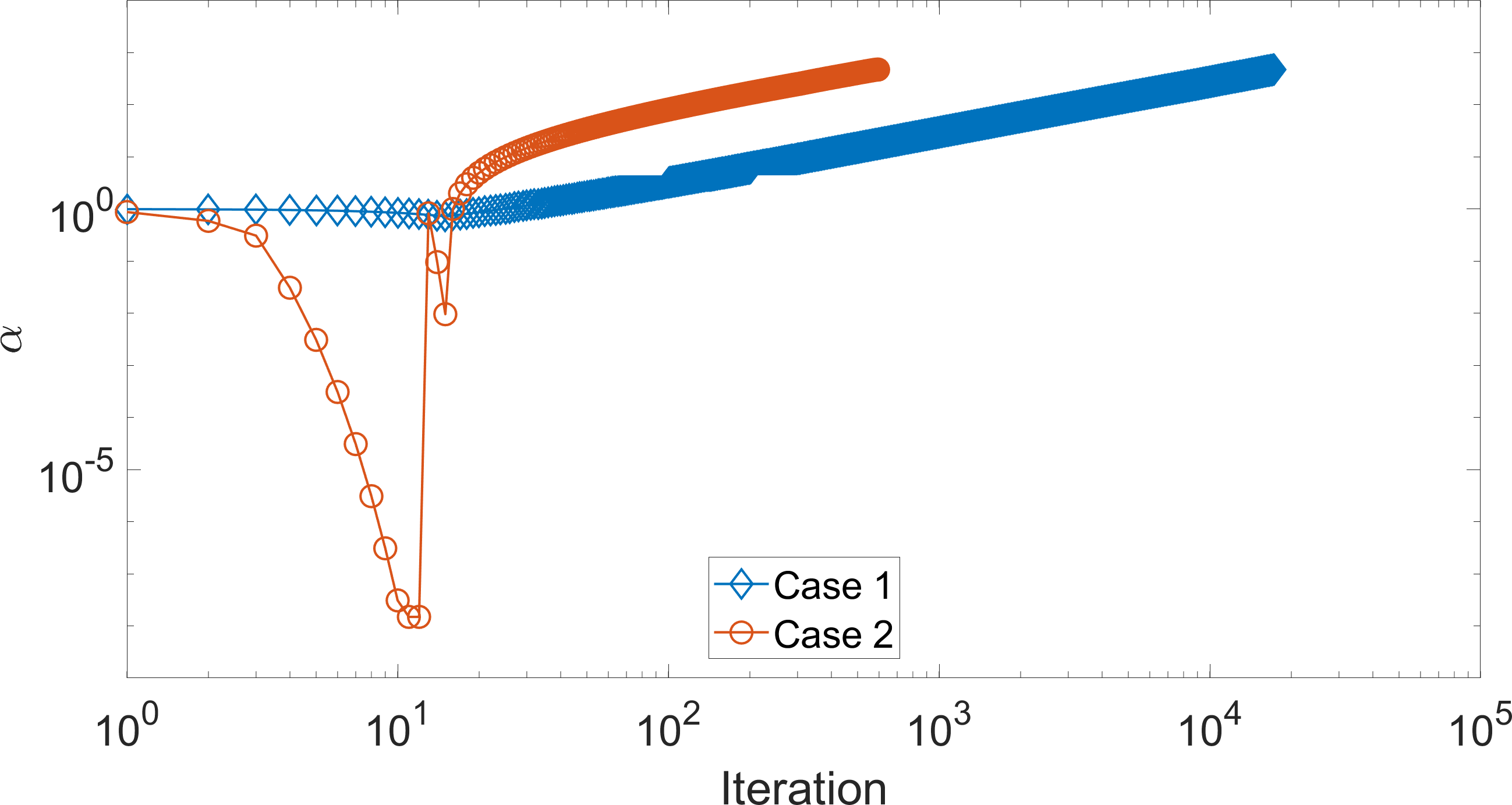}\\[2.5pt]
		\includegraphics[width = 0.49\linewidth]{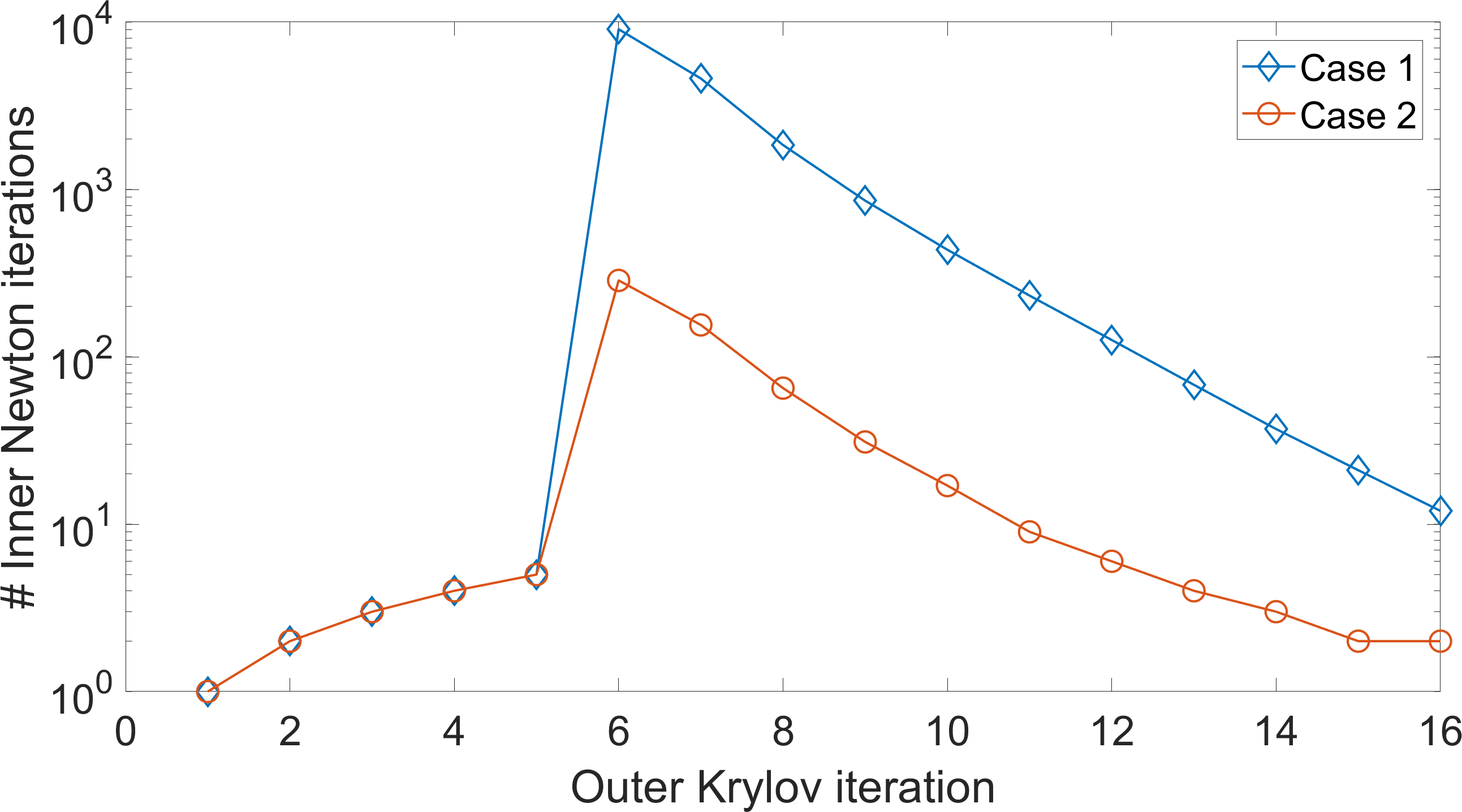}
		\caption{For each Newton iteration, we plot the point $(\alpha_k, \left\|Ax_k -
				b\right\|)$ to see where it lies with respect to the discrepancy curve. The top left
				figure corresponds to case 1, the top right figure to case 2.
				Middle left: the value of the step size used in each iteration. Middle right:
				the value of the regularization parameter in each iteration.
				Bottom: the number of inner Newton iterations per outer Krylov iteration.}
		\label{fig:rml}
\end{figure}

\begin{table}
		\centering
		\begin{tabular}{c||c|c|c}
				& \# Krylov iterations & \# Newton iterations & $\alpha$ \\ \hline\hline
				PNTM -- case 1 & $16$ ($< 1$) & $16772$ ($432$) & $469.0143$ ($5.98$) \\ \hline
				PNTM -- case 2 & $16$ ($< 1$) & $576$ ($14$)    & $469.0144$ ($5.98$) \\ \hline
				GBiT           & $32$ ($< 1$) & $\cdot$         & $469.3934$ ($5.97$)
		\end{tabular}
		\caption{Average number of iterations for the 1000 runs of the experiment and
				the standard deviation (rounded). The number of outer iterations corresponds
				to the dimension of the constructed Krylov subspace, whereas the number of
				inner iterations is the total number of Newton iterations during all the outer
				iterations. Because both methods converge to the same
				solution, the same value for $\alpha$ is found in each run, but for all the
				different random matrices its value turns out to be quite similar, hence
				the low standard deviation.}
		\label{tab:rml}
\end{table}


\subsection{Computed tomography}
As a second numerical experiment, we consider x-ray computed tomography. Here,
the goal is to reconstruct the attenuation factor of an object based on the loss
of intensity in the x-rays after they passed through the object. Classically, the
reconstruction is done using analytical methods based on the Fourier and Randon
transformations \cite{mallat2009}. In the last decades interest has grown in algebraic reconstruction
methods due to their flexibility when it comes to incorporating prior knowledge and
handling limited data. Here, the problem is written as a linear system $Ax = b$, where
$x$ represents the attenuation of the object in each pixel, the right-hand side $b$
is related to the intensity measurements of the x-rays and $A$ is a projection matrix.
The precise structure of $A$ depends on the experimental set-up, but it is typically
very sparse. For more information we refer to \cite{joseph1982, hansen2010, siltanen2012}.
We also do not construct the matrix $A$ explicitly, but use the ASTRA toolbox
\cite{aarle2015, aarle2016} in order to calculate the matrix vector products on-the-fly
using their GPU implementation \cite{palenstijn2011}.

As a test image we take the modified Shepp--Logan phantom of size $512\times 512$
and take $720$ projection angles in $[0, \pi[$, which corresponds to a matrix $A$
of size $(720\cdot 512)\times(512\cdot 512)$. Similar to the previous experiments
we add $10\%$ noise to the exact right hand size (resulting here in $\varepsilon =
4.3513\snot{3}$), but we will only calculate the PNTM reconstruction using the larger
step size from \hypref{corollary}{cor:gamma2}. We also calculate the reconstruction
using GBiT and the simultaneous iterative reconstruction technique (SIRT) \cite{gregor2008}.
The latter is a widely used fixed point iteration method for tomographic reconstructions
based on the following recursion:
\[
		x_{k + 1} = x_k + CA^TR\left(b - Ax_k\right).
\]
Here, $R$ and $C$ are diagonal matrices whose elements are the inverse row and
column sums, i.e. $r_{ii} = 1/\sum_ia_{ij}$ and $c_{jj} = 1/\sum_ia_{ij}$. It
can also be shown that this algorithm converges to the solution of the following
weighted least squares problem:
\[
		x^* = \argmin_{x\in\mbbR^n}\left\|Ax- b\right\|_R^2
\]
Note that, on the one hand, just like PNTM or GBiT, each SIRT iteration requires
one multiplication with $A$ and one with $A^T$. On the other hand, it does not need
to construct and store a basis for the Krylov subspace, so it is
computationally less expensive and requires much less memory -- two main advantages 
of the method.

The reconstructions are shown in \hypref{figure}{fig:ct1}, with further details in
\hypref{Figure}{fig:ct2} and \hypref{table}{tab:ct}. Here, we used $tol = 1\snot{-3}$
for the PNTM and GBiT stopping criterion and stopped the SIRT iterations once the
residual was smaller than the discrepancy tolerance $\varepsilon$. Furthermore, because
the 2-norm is not always a good measure for how closely two images visually resemble
each other, we also consider the structural similarity index (SSIM)\cite{wang2004}.
For two images $x$ and $y$ and default values $C_1 = 0.01^2$ and $C_2 = 0.03^2$,
this index is given by:
\[
		SSIM(x, y) = \frac{\left(2\mu_x\mu_y + C_1\right)\left(2\sigma_{xy} + C_2
				\right)}{\left(\mu_x^2 + \mu_y^2 + C_1\right)\left(\sigma_x^2 +
				\sigma_y^2 + C_2\right)}.
\]
Here, $\mu_x$ and $\mu_y$ are the mean intensity of the images, $\sigma_x$ and
$\sigma_y$ their standard deviation and $\sigma_{xy}$ the covariance. This index
lies between $0$ and $1$ and the lower its value, the better the image $x$ resembles
the reference image $y$.

When we look at the results, we see that there is little difference between the
errors of the reconstructions, but that SIRT has a much larger SSIM. When looking
at the reconstructed images, we see see that this images is indeed smoother than
the others. Because SIRT is a stationary method, it also needs more iterations
than PNTM and GBiT, which are both Krylov methods. Similarly as with the previous
experiment, however, we see that GBiT needs almost twice as many Krylov iterations
as PNTM. When we look at \hypref{figure}{fig:ct2} we see that while the value
for the regularization parameter stagnates at a similar pace, PNTM more
quickly minimizes the value of $\wt{F}$.

\begin{figure}
		\hspace*{0.32\linewidth}\hspace*{1pt}
		\includegraphics[width = 0.32\linewidth]{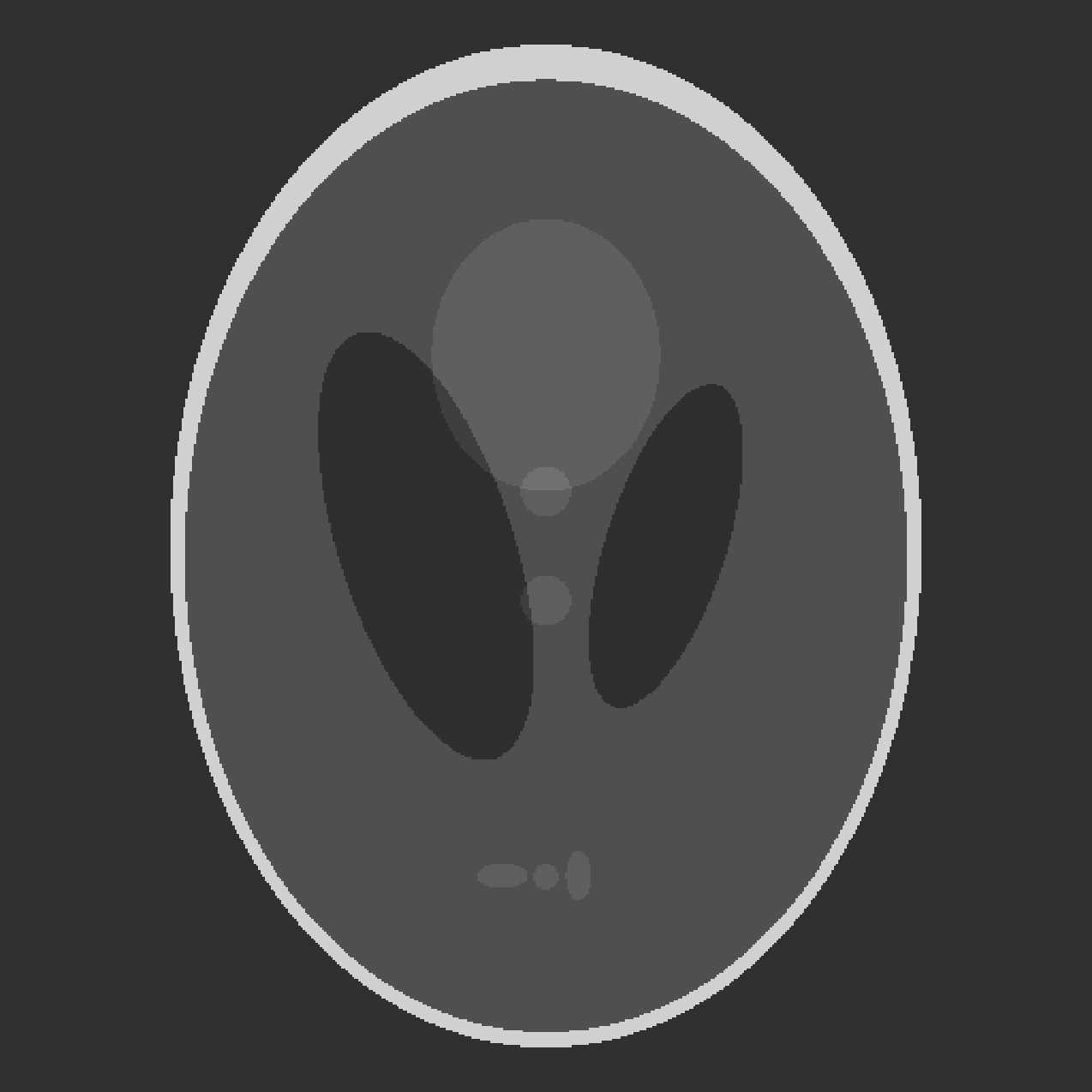}\hspace{1pt}
		\includegraphics[height = 0.32\linewidth]{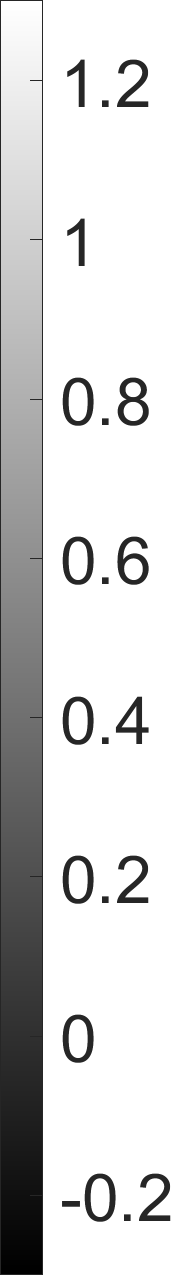}\\[2.5pt]
		\includegraphics[width = 0.32\linewidth]{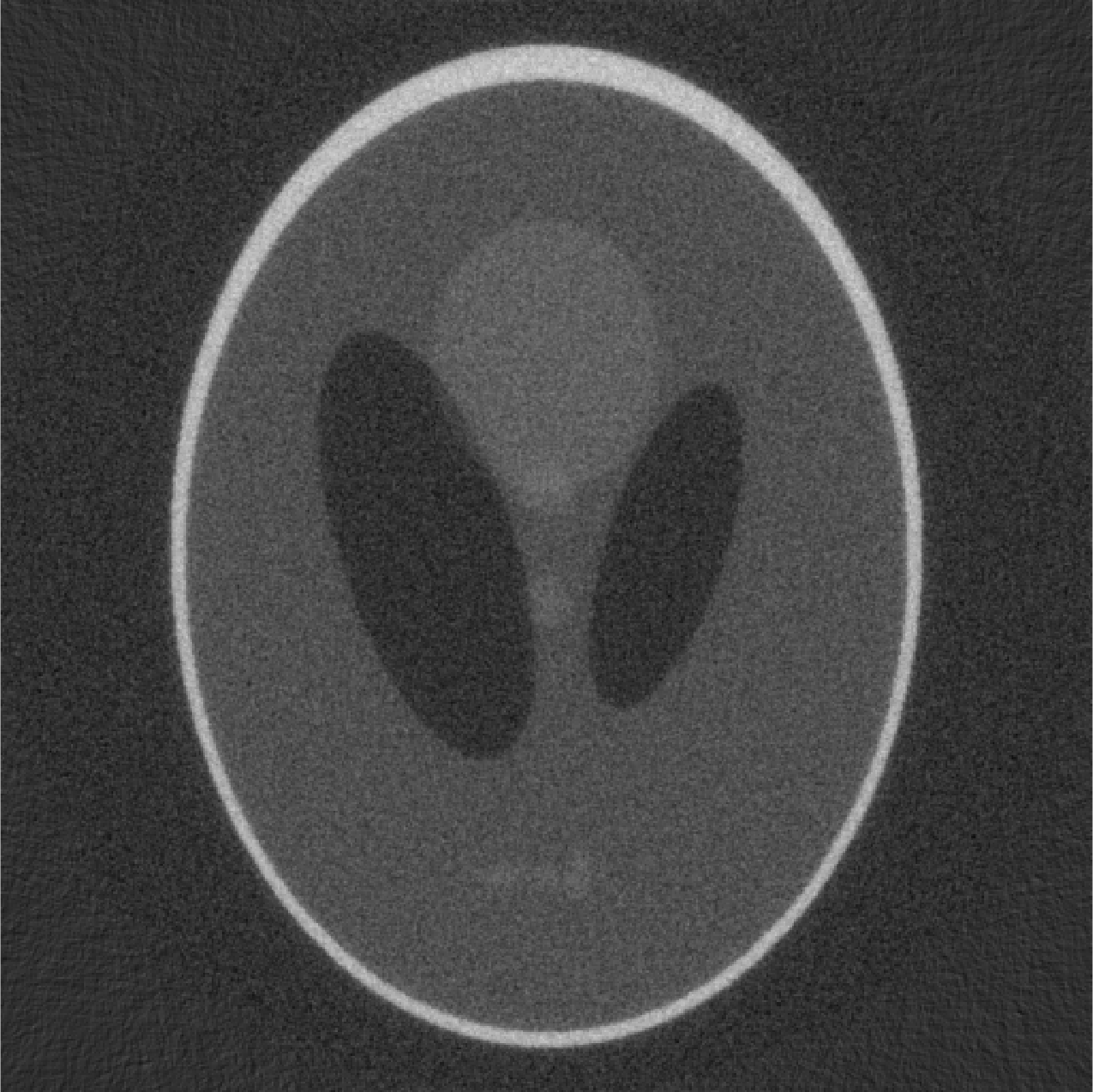}\hspace{1pt}
		\includegraphics[width = 0.32\linewidth]{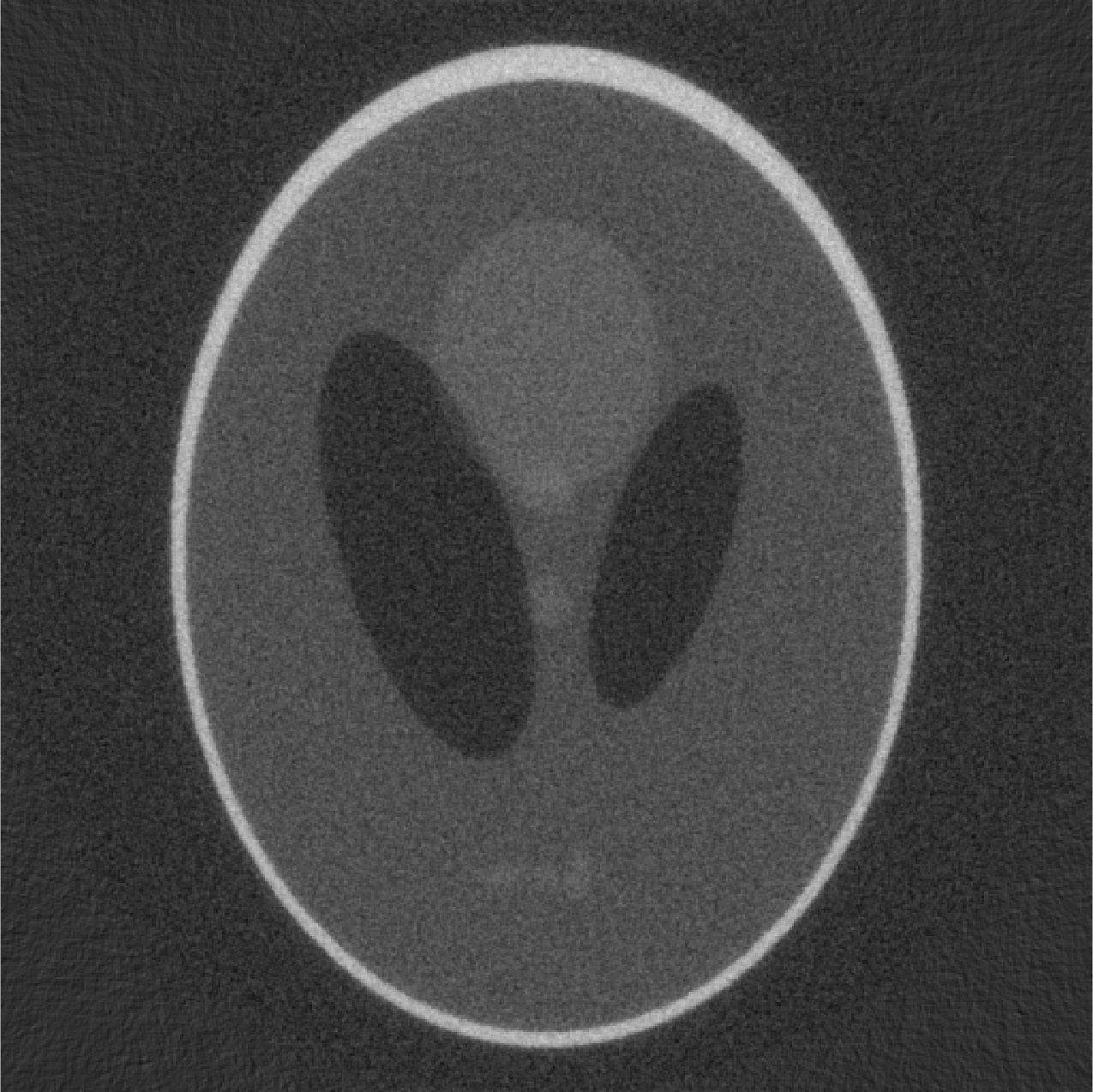}\hspace{1pt}
		\includegraphics[width = 0.32\linewidth]{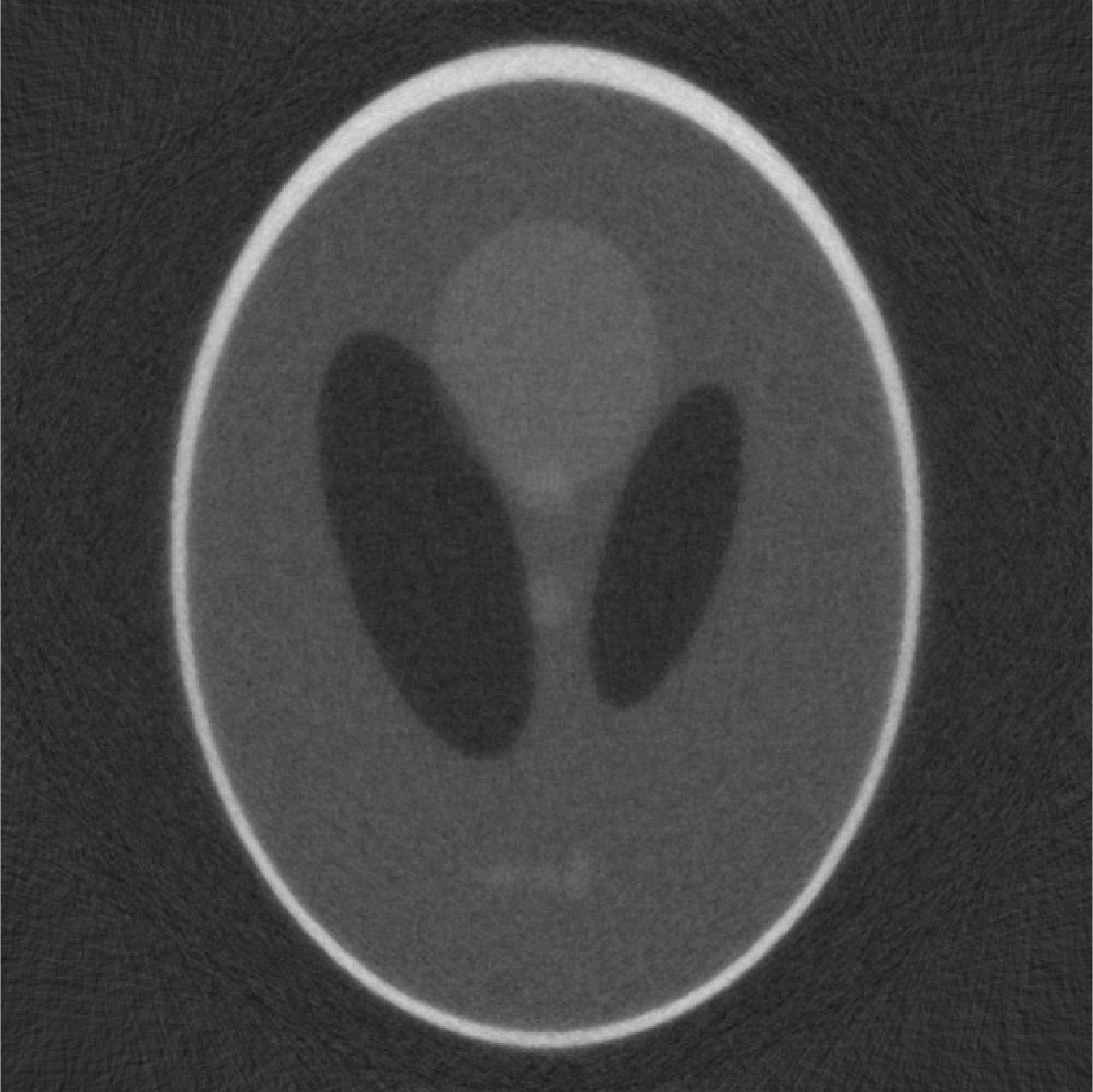}
		\caption{Top: original Shepp-Logan phantom with values in $[0, 1]$. Bottom:
				from left to right the PNTM, GBiT and SIRT reconstructions with values
				in $[-0.2074, 1.0889]$, $[-0.2071, 1.0899]$ and $[-0.1477, 1.1078]$
				respectively. Here, all images are shown on a colorscale $[-0.3, 1.3]$.}
		\label{fig:ct1}
\end{figure}

\begin{figure}
		\centering
		\includegraphics[width = 0.49\linewidth]{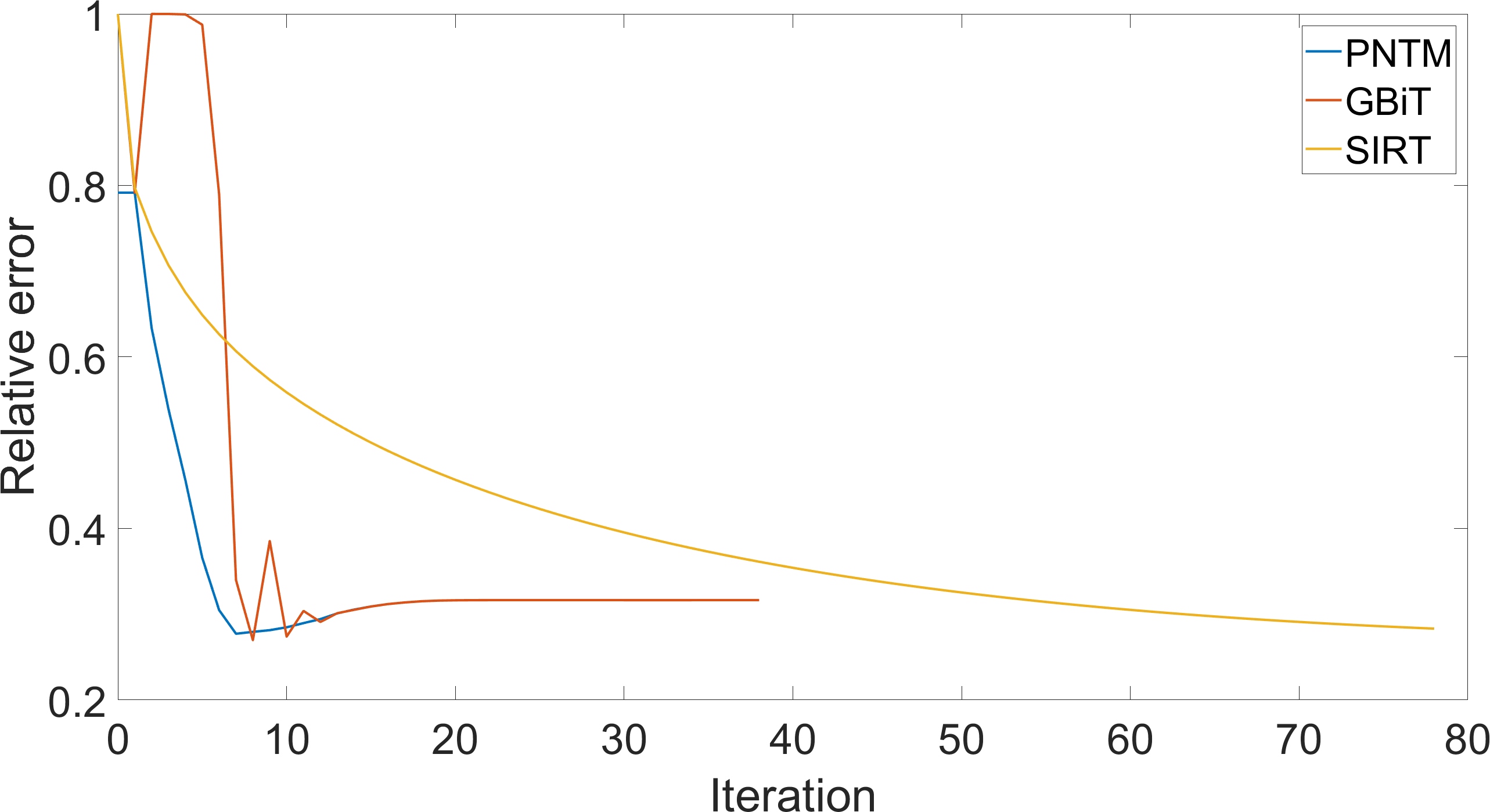}\\[2.5pt]
		\includegraphics[width = 0.49\linewidth]{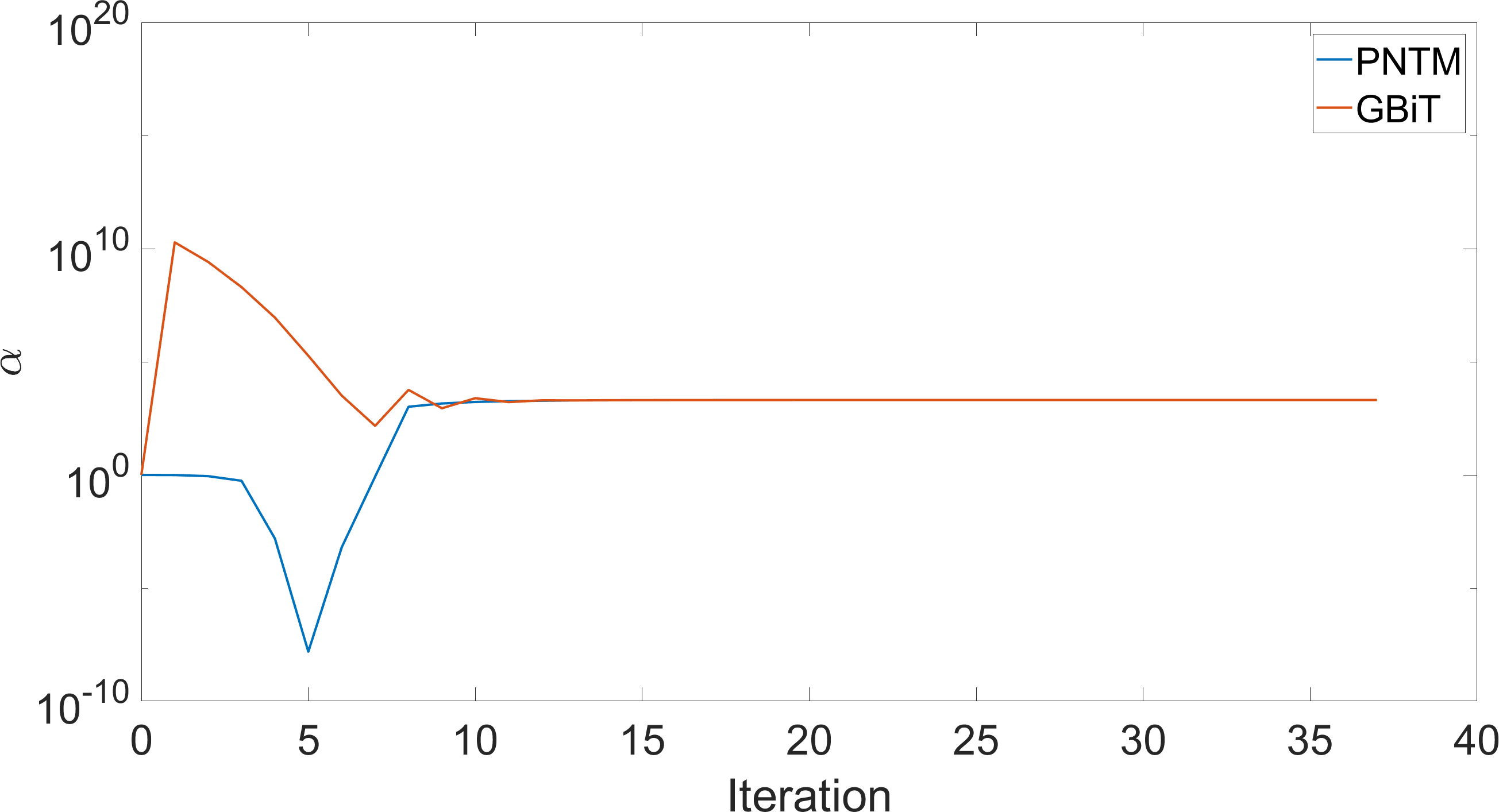}\hspace{2.5pt}
		\includegraphics[width = 0.49\linewidth]{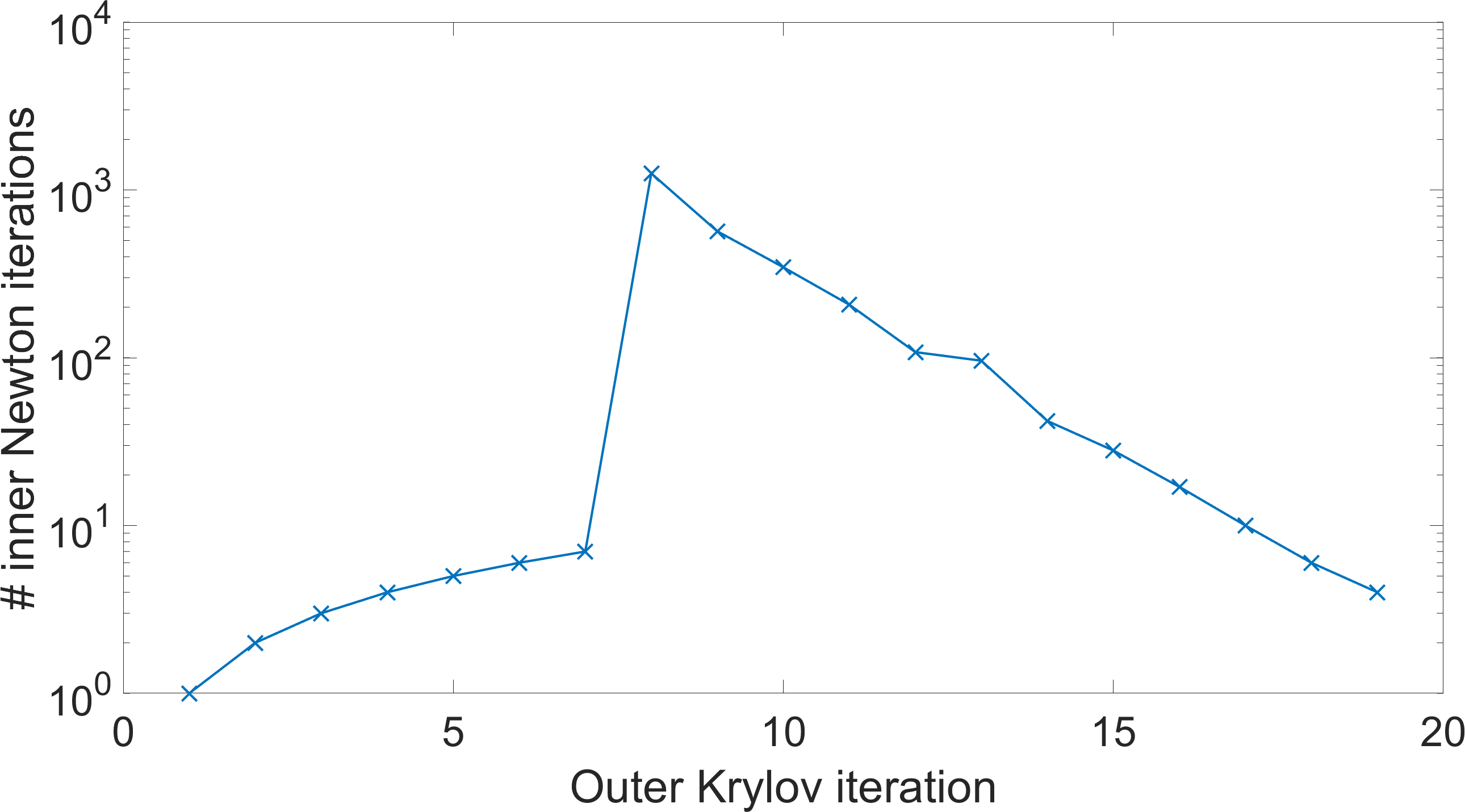}\\[2.5pt]
		\includegraphics[width = 0.49\linewidth]{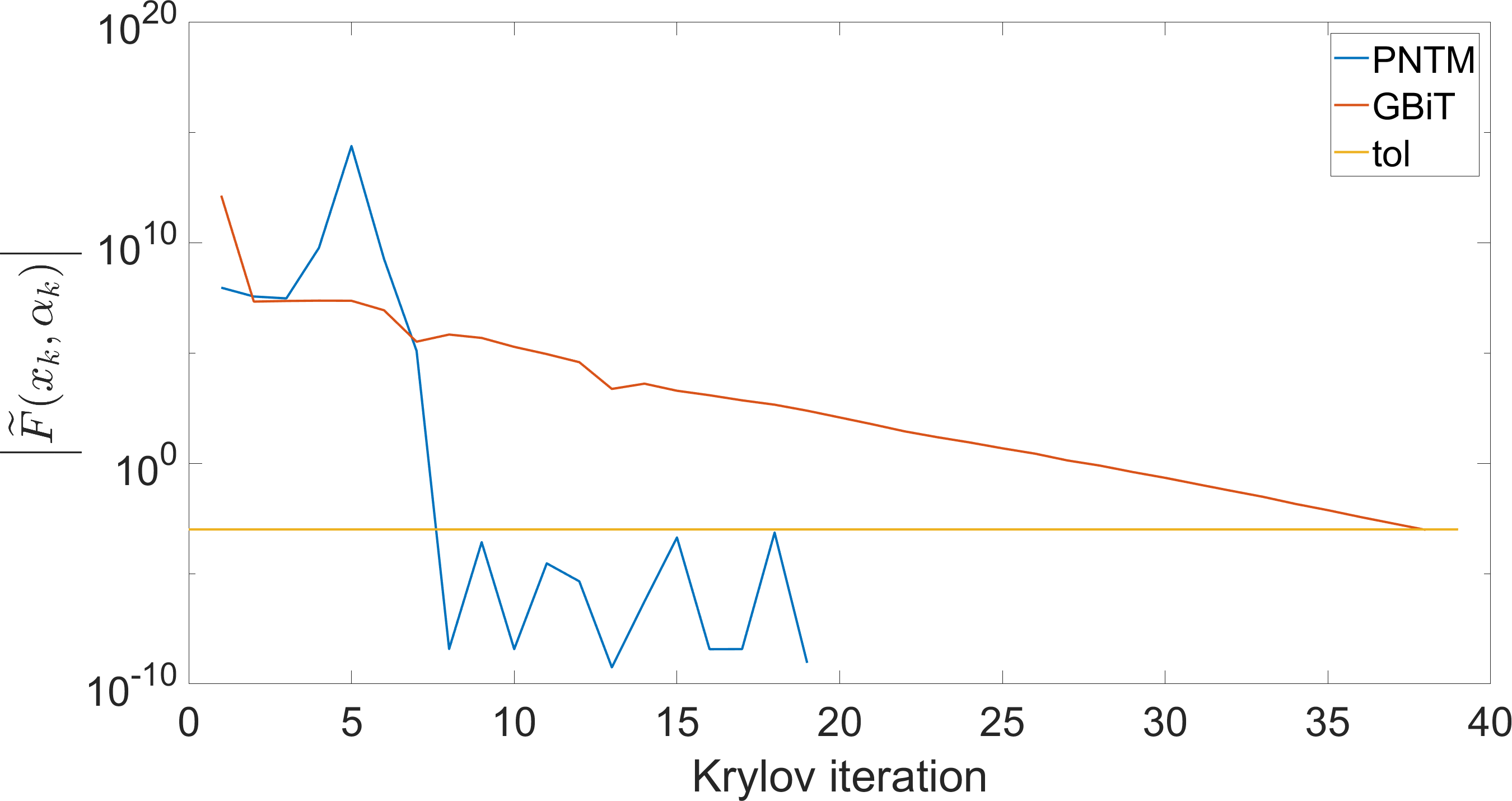}\hspace{2.5pt}
		\includegraphics[width = 0.49\linewidth]{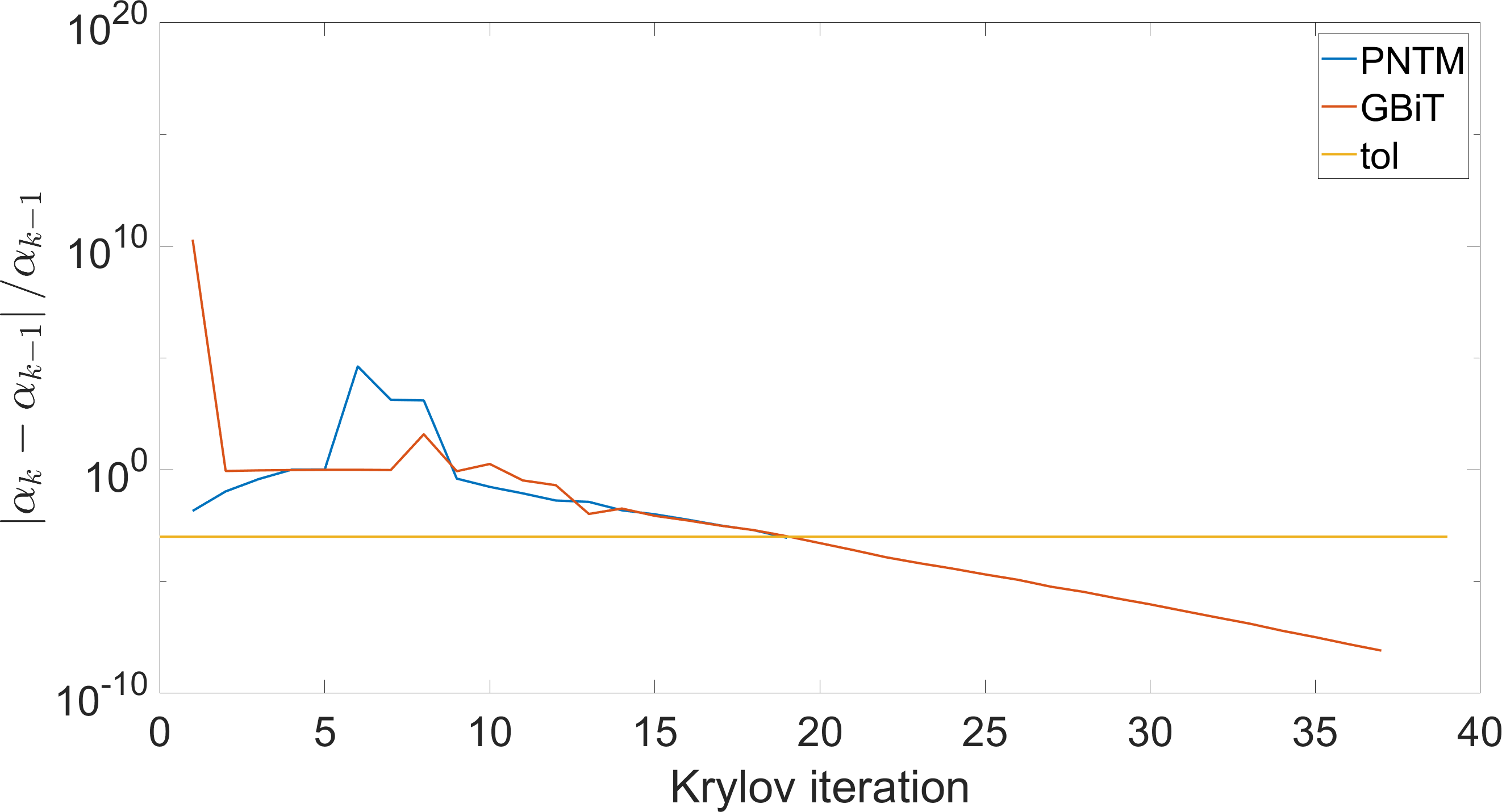}
		\caption{Top: relative error in each iteration. Middle left: value of the
				regularization parameter in each iteration. Middele right: number
				of inner Newton iterations in each outer Krylov iterations for PNTM.
				Bottom: the two parts of the stopping criterion for PNTM and GBiT.}
		\label{fig:ct2}
\end{figure}

\begin{table}
		\centering
		\begin{tabular}{c||c|c|c|c|c}
				     & \# Iterations & Relative error & Residual         & SSIM     & $\alpha$         \\ \hline\hline
				PNTM & $19$ ($2714$) & $0.3159$       & $4.3513\snot{3}$ & $0.2507$ & $2.0399\snot{3}$ \\
				GBiT & $38$          & $0.3164$       & $4.3513\snot{3}$ & $0.2499$ & $2.0413\snot{3}$ \\
				SIRT & $78$          & $0.2832$       & $4.3443\snot{3}$ & $0.4117$ & $\cdot$
		\end{tabular}
		\caption{Details from the CT reconstructions. The Krylov method PNTM and
				GBiT require less iterations than SIRT, but again PNTM needs less
				iterations than GBiT. While the relative error is very similar, the SIRT
				reconstruction has a much larger SSIM. The total number of inner Newton
				iterations for PNTM is mentioned in parentheses.}
		\label{tab:ct}
\end{table}


\subsection{Suite sparse matrix collection}
As a final experiment we take the 26 matrices $A\in\mbbR^{m\times n}$ from the
``SuiteSparse Matrix Collection'' corresponding to a least squares problem \cite{davis2011}.
For each matrix we generate a solution vector $x_{ex}\in\mbbR^n$ with entries
$x_{ex, i} = \sin(ih)$ for $h = 2\pi/(n + 1)$, calculate the right hand side
$b_{ex} = Ax_{ex}\in\mbbR^m$ and add $10\%$ noise. We then solve the resulting
inverse problem with PNTM, GBiT and priorconditionned CGLS (CGLS-PC). Again, we
use $tol = 1\snot{-3}$ for PNTM and GBiT and only consider the step size
from \hypref{corollary}{cor:gamma2}. The CGLS iterations are stopped once the residual
is smaller than $\varepsilon$. We also limit the maximum number of (outer) Krylov iterations
to $100$ and the number of inner Newton iterations for PNTM to $1000$ (\hypref{algorithm}
{alg:pntm} \hypref{line}{alg:pntm:inneriter}). Furthermore, because the $x_{ex}$ is a sine
wave, the Tikhonov problem in its standard form will result in poor reconstructions.
We therefore consider the regularization matrix 
\begin{equation}\label{eq:regmatrix}
		L = \begin{pmatrix} -1 & 1 \\ & -1 & 1 \\ && \ddots & \ddots \\
				&&& -1 & 1 \\ &&&& -1 \end{pmatrix}\in\mbbR^{n\times n},
\end{equation}
which can be seen as placing a smoothness condition on the derivative. We then
solve the problem using the transformation \eqref{eq:stdtransf}. Finally, we always
start the iterations from $\alpha_0 = 1$ and $x_0 = 0$ for CGLS.

The results are listed in \hypref{table}{tab:ssm}, where the relative discrepancy,
the relative error and the relative residue are given by
\[
		\frac{\varepsilon}{\left\|b\right\|},\qquad\text{   }\qquad
				\frac{\left\|x - x_{ex}\right\|}{\left\|x_{ex}\right\|}\qquad\text{and}\qquad
				\frac{\left\|Ax - b\right\|}{\left\|b\right\|}
\]
respectively with $x$ the reconstruction found by the algorithm. Here, we see that while
all methods find a reconstruction with a similar relative error, there are a number of
important differences. First of all note that it is logical that the priorconditionned
CGLS approach requires the least Krylov iterations. This is because the iterations are
stopped when the residual is smaller than $\varepsilon$. It is, however, only at this point
that the other two methods start to produce good values for the regularization parameter.
Then again, due to the presence of the regularization parameter, PNTM and GBiT can be
seen as more flexible. Also note that the regularization parameter $\alpha$ is chosen
by PNTM and GBiT such that the residual matches the discrepancy $\varepsilon$. In the
results we can see, however, that the PNTM method has only converged in a few cases.
It turns out that the $1000$ inner Newton iterations are insufficient for the method
to converge in the constructed Krylov subspace. This is why the total number of
Newton iterations is close to $10000$ and the relative residual does not equal the
relative discrepancy. Increasing the maximum number of inner Newton iterations could
in theory solve this issue. However, this also means that computational cost of the
method increases.

\begin{sidewaystable}
		\scalebox{0.7}{
				\begin{tabular}{l||rrrr|c|rrrrr|rrrr||rrr}
						&&&&&& \multicolumn{5}{c|}{PNTM} & \multicolumn{4}{c|}{GBiT} & \multicolumn{3}{c}{CGLS-PC} \\
						& $m$ & $n$ & \#nnz & cond. & rel. discrp. & rel. err. & rel. res. & $\alpha$ & \#K & \#N & rel. err.
								& rel. res. & $\alpha$ & \#K. & rel. err. & rel. res. & \#K \\
						\hline\hline
						abb313 & $313$ & $176$ & $1,557$ & $1.8\snot{+18}$ & $0.1001$ & $0.2369$ & $0.0989$ & $5.15\snot{+1}$ & $100$ &
								$92036$ & $0.2652$ & $0.1001$ & $6.96\snot{+1}$ & $23$ & $0.1499$ & $0.0992$ & $10$ \\
						ash85  & $85$ & $85$ & $523$ & $4.6\snot{+2}$ & $0.1005$ & $0.0665$ & $0.0887$ & $5.66\snot{+1}$ & $100$ &
								$97006$ & $0.0843$ & $0.1005$ & $1.64\snot{+2}$ & $15$ & $0.0658$ & $0.0957$ & $4$ \\
						ash219 & $219$ & $85$ & $438$ & $3.0$ & $0.1002$ & $0.0521$ & $0.0943$ & $4.33\snot{+1}$ & $100$ &
								$97006$ & $0.0648$ & $0.1002$ & $6.23\snot{+1}$ & $12$ & $0.0523$ & $0.0973$ & $4$ \\
						ash292 & $292$ & $292$ & $2,208$ & $1.2\snot{+18}$ & $0.0996$ & $0.0685$ & $0.0861$ & $8.96\snot{+1}$ & $100$ &
								$96010$ & $0.0518$ & $0.0996$ & $3.28\snot{+3}$ & $16$ & $0.0472$ & $0.0959$ & $5$ \\
						ash331 & $331$ & $104$ & $662$ & $3.1$ & $0.0989$ & $0.0340$ & $0.0989$ & $9.98\snot{-1}$ & $7$ &
								$22$ & $0.0284$ & $0.0989$ & $3.15\snot{+1}$ & $28$ & $0.0426$ & $0.0971$ & $8$ \\
						ash608 & $608$ & $188$ & $1,216$ & $3.4$ & $0.0994$ & $0.0279$ & $0.0931$ & $4.27\snot{+1}$ & $100$ &
								$96010$ & $0.0428$ & $0.0994$ & $2.25\snot{+2}$ & $13$ & $0.0340$ & $0.0993$ & $5$ \\
						ash958 & $958$ & $292$ & $1,916$ & $3.2$ & $0.0994$ & $0.0215$ & $0.0942$ & $5.01\snot{+1}$ & $100$ &
								$95015$ & $0.0280$ & $0.0994$ & $3.92\snot{+2}$ & $18$ & $0.0230$ & $0.0988$ & $6$ \\
						Delor64K & $64,719$ & $1,785,345$ & $652,140$ & $\cdot$ & $0.0996$ & $0.3342$ & $0.1008$ & $1.00\snot{+0}$ & $16$ &
								$106$ & $0.3396$ & $0.0996$ & $6.70\snot{+3}$ & $52$ & $0.3312$ & $0.0995$ & $20$ \\
						Delor295K & $295,734$ & $1,823,928$ & $2,401,323$ & $\cdot$ & $0.0996$ & $0.0209$ & $0.0997$ & $1.00\snot{+0}$ & $16$ &
								$106$ & $0.0246$ & $0.0996$ & $3.38\snot{+4}$ & $66$ & $0.0162$ & $0.0996$ & $18$ \\
						Delor338K & $343,236$ & $887,058$ & $4,211,599$ & $\cdot$ & $0.0995$ & $0.0111$ & $0.0978$ & $1.08\snot{+0}$ & $100$ &
								$92036$ & $0.0043$ & $0.0995$ & $5.65\snot{+6}$ & $27$ & $0.0031$ & $0.0995$ & $10$ \\
						ESOC & $327,062$ & $37,830$ & $6,019,939$ & $\infty$ & $0.0995$ & $0.0586$ & $0.0985$ & $2.01\snot{-8}$ & $100$ &
								$74215$ & $0.0591$ & $0.0995$ & $1.34\snot{+14}$ & $100$ & $0.0225$ & $0.0995$ & $53$ \\
						illc1033 & $1,033$ & $320$ & $4,719$ & $1.9\snot{+4}$ & $0.0993$ & $0.0404$ & $0.0976$ & $1.50\snot{+1}$ & $65$ &
								$58486$ & $0.0508$ & $0.0993$ & $2.31\snot{+1}$ & $14$ & $0.0406$ & $0.0991$ & $6$ \\
						illc1850 & $1,850$ & $712$ & $8,636$ & $1.4\snot{+3}$ & $0.0993$ & $0.0135$ & $0.0972$ & $1.85\snot{+1}$ & $100$ &
								$93028$ & $0.0232$ & $0.0993$ & $4.65\snot{+1}$ & $19$ & $0.0172$ & $0.0989$ & $9$ \\
						landmark & $71,952$ & $2,704$ & $1,146,848$ & $\infty$ & $0.0995$ & $0.0115$ & $0.0993$ & $4.74\snot{+1}$ & $100$ &
								$77185$ & $0.0120$ & $0.0995$ & $4.44\snot{+2}$ & $55$ & $0.0134$ & $0.0995$ & $37$ \\
						Maragal\textunderscore 1 & $32$& $14$ & $234$ & $4.6\snot{+16}$ & $0.0989$ & $0.2048$ & $0.0989$ & $2.80\snot{+0}$ & $8$ &
								$60$ & $0.2048$ & $0.0989$ & $2.80\snot{+0}$ & $9$ & $0.1784$ & $0.0936$ & $4$ \\
						Maragal\textunderscore 2 & $555$ & $350$ & $4,357$ & $2.9\snot{+47}$ & $0.0982$ & $0.0234$ & $0.0942$ & $2.79\snot{+1}$ & $100$ &
								$94021$ & $0.0213$ & $0.0982$ & $9.78\snot{+1}$ & $16$ & $0.0191$ & $0.0977$ & $7$ \\
						Maragal\textunderscore 3 & $1,690$ & $860$ & $18,391$ & $1.5\snot{+47}$ & $0.0993$ & $0.0194$ & $0.0945$ & $3.72\snot{+1}$ & $100$ &
								$94021$ & $0.0225$ & $0.0993$ & $4.62\snot{+2}$ & $20$ & $0.0136$ & $0.0991$ & $7$ \\
						Maragal\textunderscore 4 & $1,964$ & $1,034$ & $26,719$ & $6.1\snot{+33}$ & $0.0996$ & $0.0218$ & $0.0936$ & $4.47\snot{+1}$ & $100$ &
								$96010$ & $0.0321$ & $0.0996$ & $9.87\snot{+2}$ & $17$ & $0.0123$ & $0.0996$ & $5$ \\
						Maragal\textunderscore 5 & $4,654$ & $3,320$ & $93,091$ & $7.4\snot{+31}$ & $0.0994$ & $0.0192$ & $0.0926$ & $5.62\snot{+1}$ & $100$ &
								$95015$ & $0.0328$ & $0.0994$ & $8.95\snot{+3}$ & $17$ & $0.0147$ & $0.0985$ & $6$ \\
						Maragal\textunderscore 6 & $21,255$ & $10,152$ & $537,694$ & $3.3\snot{+33}$ & $0.0995$ & $0.0153$ & $0.0953$ & $7.69\snot{+1}$ & $100$ &
								$95015$ & $0.0174$ & $0.0995$ & $6.09\snot{+4}$ & $20$ & $0.0072$ & $0.0994$ & $6$ \\
						Maragal\textunderscore 7 & $46,845$ & $26,564$ & $1,200,537$ & $\infty$ & $0.0996$ & $0.0160$ & $0.0960$ & $8.26\snot{+1}$ & $100$ &
								$74215$ & $0.0066$ & $0.0996$ & $2.70\snot{+3}$ & $77$ & $0.0119$ & $0.0995$ & $41$ \\
						Maragal\textunderscore 8 & $33,212$ & $75,077$ & $1,308,415$ & $\infty$ & $0.0994$ & $0.0211$ & $0.0930$ & $5.05\snot{+1}$ & $100$ &
								$85105$ & $0.0036$ & $0.0994$ & $3.62\snot{+4}$ & $100$ & $0.0064$ & $0.0994$ & $20$ \\
						Rucci1 & $1,977,885$ & $109,900$ & $7,791,168$ & $\cdot$ & $0.0995$ & $0.0203$ & $0.1006$ & $1.00\snot{+0}$ & $6$ &
								$16$ & $0.0248$ & $0.0995$ & $1.38\snot{+2}$ & $28$ & $0.0019$ & $0.0994$ & $9$ \\
						sls & $1,748,122$ & $62,729$ & $6,804,304$ & $\cdot$ & $0.0995$ & $0.0068$ & $0.0992$ & $1.64\snot{+0}$ & $100$ &
								$89065$ & $0.0028$ & $0.0995$ & $2.59\snot{+6}$ & $33$ & $0.0023$ & $0.0995$ & $14$ \\
						well1033  & $1,033$ & $320$ & $4,732$ & $1.7\snot{+2}$ & $0.0999$ & $0.0188$ & $0.0988$ & $1.26\snot{+1}$ & $36$ &
								$27098$ & $0.0229$ & $0.0999$ & $1.83\snot{+1}$ & $19$ & $0.0154$ & $0.0998$ & $8$ \\
						well1850  & $1,850$ & $712$ & $8,755$ & $1.1\snot{+2}$ & $0.0998$ & $0.0170$ & $0.0955$ & $2.37\snot{+1}$ & $100$ &
								$95015$ & $0.0565$ & $0.0998$ & $1.80\snot{+2}$ & $14$ & $0.0339$ & $0.0982$ & $6$
				\end{tabular}
		}
		\caption{Details of the 26 matrices and the PNTM, GBiT and CGLS-PC reconstructions.
				\#K indicates the number of Krylov iterations and \#N the total number of
				inner Newton iterations for PNTM. Because we limited the number this number,
				the PNTM has trouble satisfying the stopping criterion, despite the fact that
				reconstruction has similar quality as the other methods.}
		\label{tab:ssm}
\end{sidewaystable}


\section{Conclusions \& remarks}\label{sec:concl}
In this paper we introduced two different numerical methods: Newton on the Tikhonov-
Morozov system (NTM) and projected Newton on the Tikhonov-Morozov system (PNTM). We
derived the NTM method based on theoretical results and illustrated two difficulties:
the estimated step size and the computational cost. In order to reduce the computational
cost we projected the problem onto a low dimensional Krylov subspace.
The small estimate for the step size, however, remains an issue.

In the numerical experiments it is important to note the difference between GBiT
(and by extension GAT) and PNTM. While both methods solve the inverse problem in
increasingly larger Krylov subspaces, the value that is minimized in each Krylov
subspace and the way the regularization parameter is updated are different. GBiT
solves the projected Tikhonov normal equations in each Krylov subspace using a fixed
regularization parameter and only afterwards updates the regularization parameter
for the next Krylov iteration. This can be seen as alternating between minimizing
$\wt{F}_1$ using a Krylov method and minimizing $\wt{F}_2$ using the secant method.
The PNTM method minimizes both values simultaneously in the Krylov subspace using
Newton's method and only expands the Krylov subspace if the value for the regularization
parameter has not stagnated yet. Our numerical experiments seem to indicate that
the alternating approach of GBiT is less efficient than the simultaneous update
approach of PNTM. This however assumes that the number of inner Newton iterations
for PNTM is high enough for them to converge. As a result of the small estimate
for the step size we currently use, this may take too many iterations to be a
viable alternative. Improving the choice of the step size -- possibly using a
backtracking approach -- is therefore necessary in order to improve this method.


\section*{Acknowledgments}
The authors wish to thank the Department of Mathematics and Computer Science,
University of Antwerp, for financial support.


\bibliographystyle{siamplain}
\bibliography{references}


\end{document}